\documentclass[11pt]{article}
\usepackage[margin=1in]{geometry}
\usepackage[flushleft]{threeparttable}
\usepackage{amsthm}
\usepackage{amsfonts}
\usepackage{amsmath}
\numberwithin{equation}{section}
\usepackage{multirow}
\usepackage{stmaryrd}
\usepackage{graphicx}
\usepackage{subfigure}
\usepackage{indentfirst}

\usepackage{enumitem}
\usepackage{tabularx}
  \newcolumntype{R}{>{\raggedleft\arraybackslash}X}
  \newcolumntype{L}{>{\raggedright\arraybackslash}X}
  \newcolumntype{C}{>{\centering\arraybackslash}X}
\usepackage{booktabs}
\usepackage{rotating}

\usepackage{caption}
\usepackage[cmintegrals]{newpxmath}

\usepackage{dsfont}                
  \newcommand{\IQ}{\ensuremath\mathds{Q}}
  \newcommand{\IR}{\ensuremath\mathds{R}}
  \newcommand{\IP}{\ensuremath\mathds{P}}
  \newcommand{\IV}{\ensuremath\mathds{V}}         
\usepackage{bm}                      
\renewcommand*{\vec}[1]{{\boldsymbol{#1}}}
\DeclareMathAlphabet{\mathbfsf}{\encodingdefault}{\sfdefault}{bx}{n}

\newcommand*{\setE}{\ensuremath{\mathcal{T}}}    
\newcommand*{\transpose}[1]{{#1}^\top}                     
    \renewcommand*{\div}{\vec{\nabla}\cdot}                  
\newcommand{\IZ}{\ensuremath\mathds{Z}} 
 
\usepackage{xcolor}

\newcommand{\exte}{{\rm{ext}}}
\newcommand{\inte}{{\rm{int}}}
\newcommand{\wB}{\omega^{\mathrm{B}}}
\newcommand{\wI}{\omega^{\mathrm{I}}}
\newcommand{\quand}{\quad \text{and} \quad}
\newcommand{\dd}{\mathrm{d}}

\newcommand{\DG}{\mathrm{DG}}

\newcommand{\dxDG}{\nabla^\DG\cdot}
\newcommand{\dxl}{\nabla^\mathrm{loc}\cdot}
\newcommand{\rr}{s}

\theoremstyle{definition}

\newtheorem{THM}{Theorem}[section] 

\newtheorem{PROP}[THM]{Proposition}

\newtheorem{LEM}[THM]{Lemma}

\theoremstyle{remark}
\newtheorem{remark}{Remark}
\newcommand{\revone}[1]{#1}
\newcommand{\revtwo}[1]{#1}
\newcommand{\revthr}[1]{#1}
\newenvironment{revthrblock}
  {}
  
\title{A bound-preserving Runge--Kutta discontinuous Galerkin method with compact stencils for hyperbolic conservation laws}
\author{
Chen Liu\footnote{Department of Mathematical Sciences, University of Arkansas, Fayetteville, AR 72701. Email: chenl@uark.edu.},~ 
Zheng Sun\footnote{Department of Mathematics, The University of Alabama, Tuscaloosa, AL 35487. Email: zsun30@ua.edu. Z. Sun is  supported by NSF DMS-2208391.},~
Xiangxiong Zhang\footnote{Department of Mathematics, Purdue University, West Lafayette, IN 47907. Email: zhan1966@purdue.edu. X. Zhang is  supported by NSF DMS-2208518. }
} 
\date{}

\newcommand{\cF}{\mathcal{F}}
\newcommand{\cG}{\mathcal{G}}
\newcommand{\cL}{\mathcal{L}}

\newcommand{\ns}{m}

\newcommand{\al}{a_0}
\newcommand{\nlK}{{l(K)}}
\begin{document}

\maketitle 

\noindent 
\textbf{Abstract:} 
In this paper, we develop bound-preserving techniques for the Runge--Kutta (RK) discontinuous Galerkin (DG) method with compact stencils (cRKDG method) for hyperbolic conservation laws. The cRKDG method was recently introduced in [Q.~Chen, Z.~Sun, and Y.~Xing, \emph{SIAM J. Sci. Comput.}, 46: A1327--A1351, 2024]. It enhances the compactness of the standard RKDG method, resulting in reduced data communication, simplified boundary treatments, and improved suitability for local time marching. This work improves the robustness of the cRKDG method by enforcing desirable physical bounds while preserving its compactness, local conservation, and high-order accuracy. Our method is extended from the seminal work of [X.~Zhang and C.-W.~Shu, \emph{J. Comput. Phys.}, 229: 3091--3120, 2010]. We prove that the cell average of the cRKDG method at each RK stage preserves the physical bounds by expressing it as a convex combination of three types of forward-Euler solutions. A scaling limiter is then applied after each RK stage to enforce pointwise bounds. Additionally, we explore RK methods with less restrictive time step sizes. Because the cRKDG method does not rely on strong-stability-preserving RK time discretization, it avoids its order barriers, allowing us to construct a four-stage, fourth-order bound-preserving cRKDG method. Numerical tests on challenging benchmarks are provided to demonstrate the performance of the proposed method. \\


\section{Introduction}
In this paper, we design bound-preserving techniques for the Runge--Kutta (RK) discontinuous Galerkin (DG) method with compact stencils (cRKDG method) for hyperbolic conservation laws introduced by Chen, Sun, and Xing in \cite{chen2024runge}. The proposed bound-preserving cRKDG method preserves the invariant region of the equation systems while maintaining the compactness, local conservation, and high-order accuracy of the solver. 
\par
We consider hyperbolic conservation laws in general form with an unknown $u:\Omega\times [0,T]\to \IR^m$, defined over a spatial domain $\Omega$ and a time interval $[0,T]$: 
\begin{equation}\label{eq:claws}
\partial_t u + \nabla \cdot f(u) = 0, \quad x \in \Omega \subset \IR^d, \quad t\in [0,T]. 
\end{equation}
Here $f = [f_1, f_2, \cdots, f_d]\!: B\subset \IR^m\to \IR^{m\times d}$ is the flux function that may be scalar-, vector- or matrix-valued, depending on the context.
The solution to \eqref{eq:claws} typically admits certain physical bounds and resides within a convex invariant region $B$ (also called an admissible set).  
Numerical methods that fail to preserve such bounds can result in nonphysical solutions, which may make the problem ill-posed and lead to instability during simulations.
For instance, when negative internal energy emerges, the linearized compressible Euler equations are no longer hyperbolic, and the associated initial value problem is ill-posed \cite{zhang2010positivity}.  Such violations of physical bounds in numerical solutions pose significant practical challenges. They have attracted considerable attention both for hyperbolic conservation laws \cite{zhang2010maximum,xu2014parametrized,guermond2019invariant,wu2023geometric} and in other similar contexts \cite{huang2018bound,chuenjarern2019high,sun2019entropy,du2023oscillation,liu2024optimization,liu2024simple}. This paper aims to address these issues for the cRKDG method for \eqref{eq:claws}.
\par
The DG method is a class of finite element methods that employs discontinuous piecewise polynomial spaces. It was first introduced by Reed and Hill in the 1970s for solving the steady-state transport equation \cite{reed1973triangular}. In subsequent pioneering works by Cockburn et al. \cite{rkdg1,rkdg2,rkdg3,rkdg4,rkdg5}, the DG spatial discretization was combined with the RK time stepping method for solving the hyperbolic conservation laws, resulting in the RKDG method. 
The RKDG method features many desirable properties, such as high-order accuracy, good $hp$ adaptivity, compatibility with limiters for robust capturing of shock waves, flexibility for handling complex geometry, high parallel efficiency, etc. The RKDG method has become one of the mainstream numerical methods in computational fluid dynamics \cite{cockburn2001runge}. 
\par
The cRKDG method is a variant version of the RKDG method with improved compactness. In the original RKDG method, the numerical scheme needs information from the neighboring cells for evaluating the numerical flux in each RK stage. Therefore, the stencil size of the RKDG method expands with the number of stages. To address this issue, in \cite{chen2024runge}, we propose the cRKDG method, which uses the interior trace, instead of the numerical flux, for evaluation of the cell interface terms at all inner RK stages. Thus, data communication only occurs at the final stage in each time step. Therefore, the resulting method is compact, in the sense that its stencil only involves the current mesh cell and immediate neighbors. The compactness of the cRKDG method allows for additional advantages of the method, including reduced data communication, high parallel efficiency, simple boundary treatment, suitability for local time marching, etc. The cRKDG method is known to have connections with a few other fully discrete DG schemes \cite{qiu2005discontinuous,dumbser2008unified,gassner2011explicit}. Its construction hybridizes different spatial operators within a time step, sharing similar flavors as some of the recent works \cite{xu2024sdf,chen2024sdrkdg,sun2024reducing}. 
\par
Techniques for preserving the bound of standard RKDG schemes have been extensively studied in the literature. Our research primarily follows the pioneering work of Zhang and Shu \cite{zhang2010maximum,zhang2010positivity}. We also refer to \cite{zhang2011maximum, zhang2012positivity, xu2017bound,shu2020boundpreserving} for further related developments. This stream of research utilizes the following methodology for designing conservative and high-order bound-preserving RKDG schemes. First, it is shown that under an appropriate Courant--Friedrichs--Lewy (CFL) condition, the cell averages of the forward-Euler DG solution resides within the invaraint region $B$. This property is usually referred to as the weak positivity or the weak bound-preserving property. Second, a scaling limiter is applied to squash the solution towards the bound-preserving cell averages to enforce pointwise bound. Third, the bound-preserving solver with forward-Euler method is then extended to higher temporal order by using the strong-stability-preserving (SSP) RK methods \cite{gottlieb2001strong,gottlieb2011strong}. 
\par
In this paper, we develop bound-preserving strategies for the cRKDG solvers, to embrace its numerical advantages while improving its robustness by enforcing appropriate physical bounds. The main challenge in this task comes from two aspects. First, the cRKDG method hybridizes two different spatial operators in a single time step, the original DG operator and the local derivative operator. Hence, a single building block only involving the DG operator will not be sufficient. Second, the cRKDG method is designed with RK methods in the standard Butcher form with no obvious SSP structures. Hence, it is not straightforward to take advantage of the convex property to reduce the problem of preserving the bound for a high-order-in-time method to that for a first-order-in-time method. 
\par
The main idea to circumvent the aforementioned challenge is to formulate the cRKDG method as a convex combination of three different types of forward-Euler solvers: the original DG operator, the forward-in-time local derivative operator, and the backward-in-time local derivative operator. First, we show that, in addition to the weak bound-preserving property of the forward-Euler step for the DG operator, the same property holds for the local derivative operator, both forward and backward in time. This property is proved based on what is referred to in the literature as the Lax–Friedrichs splitting property \cite{wu2015high}. Second, we demonstrate that in many RK methods, including those in Butcher form, each RK step can be expressed as a convex combination of three types of forward-Euler solvers. Consequently, each stage satisfies the weak bound-preserving property. Finally, the scaling limiter is applied after each RK stage to enforce pointwise bounds.
\par
Our new solver not only preserves the physical bound for the cRKDG method, but also maintains its compactness, the local conservation, and high-order spatial accuracy. The method we design applies to both the scalar conservation laws and hyperbolic systems. Although here, we have been particularly focusing on the compressible Euler equations for gas dynamics, the idea can be easily extended to other systems, such as the shallow water equations and the relativistic hydrodynamics (see Remark~\ref{rmk:other_systems}). We have also spent efforts on identifying RK methods that permit large time steps. Notably, since our method does not rely on the SSP-RK methods, we circumvent their order barriers and are able to construct a fourth-order bound-preserving cRKDG method with only four temporal stages. 
\par
The rest of the paper is organized as follows. In Section~\ref{sec:background}, we introduce the notations and  provide the background for the RKDG, cRKDG, and bound-preserving SSP-RKDG methods. 
In Section~\ref{sec:bound_preserving}, we prove the weak bound-preserving properties and design the bound-preserving cRKDG schemes for both scalar conservation laws and hyperbolic systems with particular focus on compressible Euler equations. 
In Section~\ref{sec:numerical_experiment}, we provide numerical tests with scalar conservation laws and Euler equations both in 1D and 2D to demonstrate the performance of the proposed method. Conclusions are given in Section~\ref{sec:conclusion}.

\section{Notations and preliminaries}\label{sec:background}
Let $\setE_h = \{K\}$ be a mesh partition of the spatial domain $\Omega$. We denote by $h_K$ the diameter of a cell $K$ and $h = \max_{K\in \setE_h} h_K$. The discontinuous piecewise polynomial space is defined as 
\begin{equation}
\IV_h = \{v \in L^2(\Omega)\!: v|_K \in [\IZ^k(K)]^\ns,~ \forall K \in \setE_h\}.
\end{equation}
Here $\IZ^k(K) = \IP^k(K)$ or $\IQ^k(K)$, where $\IP^k(K)$ is the linear space for polynomials on $K$ of degree less than or equal to $k$ ($k\geq1$), and $\IQ^k(K)$ is the linear space for polynomials on $K$ of degree less than or equal to $k$ ($k\geq1$) in each variable. For each edge $e_{\ell,K} \in \partial K$, we define $\nu_{\ell,K}$ to be the unit outer normal. 
Moreover, let $v_h^\inte$ and $v_h^\exte$ denote the interior and exterior traces of $u_h$ on $e_{\ell,K}$, respectively. 

\subsection{RKDG and cRKDG methods}
\paragraph{Semidiscrete DG method.}  
In the semidiscrete DG method for \eqref{eq:claws}, we approximate the solution $u(\cdot,t)$ by a discontinuous piecewise polynomial $u_h(\cdot,t) \in \IV_h$ and use the DG operator $\dxDG f$ to approximate the flux divergence $\nabla \cdot f$. It yields  
\begin{equation}\label{eq:semi}
    \partial_t u_h + \dxDG f(u_h) = 0.
\end{equation}
Here, $\dxDG f(\cdot):\IV_h \to \IV_h$ is the DG operator, which is defined as the following. 
\begin{align}\label{eq:dxDG}
\int_K [\dxDG f(u_h)]\cdot  v_h  \dd x=& -\int_K f(u_h) : \nabla v_h \dd x + \sum_{\ell=1}^{\nlK} \int_{e_{\ell,K}} \widehat{f\cdot\nu_{\ell,K}}\big(u_h^\inte, u_h^\exte \big) \cdot  v_h \dd l\quad \forall v_h \in \IV_h,
\end{align}
where ``:'' denotes the Frobenius inner product of two matrices and $\widehat{f\cdot\nu_{\ell,K}}\big(u_h^\inte,u_h^\exte \big)$ is the numerical flux. For instance, one can take the Lax--Friedrichs flux 
\begin{equation}
\widehat{f\cdot \nu_{\ell,K}} 
= \frac{f(u_h^\inte) + f(u_h^\exte)}{2}\cdot\nu_{\ell,K} - \frac{\alpha_{\ell,K}}{2}\left(u_h^\exte - u_h^\inte\right)
\quad\text{with}\quad \alpha_{\ell,K} = \max |\mathrm{eig}\left(\partial_u f \cdot \nu_{\ell,K}\right)|. 
\end{equation}

\paragraph{RKDG method.} For the RKDG scheme, an explicit RK time-stepping method is used to discretize \eqref{eq:semi}. Consider an explicit RK method with the Butcher tableau
\begin{equation}\label{eq:butcher}
\begin{array}{c|c}
c & A\\ \hline
~ & b\\
\end{array}, \quad 
A = (a_{ij})_{s\times s}, \quad b = (b_1,\cdots, b_s), \quad c = (c_1, \cdots, c_s)^\top,
\end{equation}
where $A$ is a lower triangular matrix, namely,  $a_{ij} = 0$ for $i\leq j$. 
Then, the fully discrete RKDG scheme with the time step size $\Delta t$ is written as 
\begin{subequations}\label{eq:rkdg}
\begin{align}
u_h^{(i)} &= u_h^n -  \Delta t\sum_{j = 1}^{i-1}  a_{ij} \dxDG f\left(u_h^{(j)}\right), \quad i = 1, 2, \cdots, s,\label{eq-rkdg1}\\
u_h^{n+1} &= u_h^n - \Delta t \sum_{i = 1}^s b_i \dxDG  f\left(u_h^{(i)} \right).\label{eq:rkdg-2}
\end{align}
\end{subequations}

In many applications, the SSP-RK methods \cite{gottlieb2001strong,gottlieb2011strong} are used for time integration. Such RK schemes can be written in Shu--Osher form as convex combinations of the forward-Euler steps. The corresponding fully discrete SSP-RKDG schemes take the following form. 
\begin{subequations} \label{rk:ssp way1}
\begin{align}
u_h^{(0)} &= u_h^n, \\
u_h^{(i)} &= \sum_{j=0}^{i-1}\left(\alpha_{i j} u_h^{(j)}-\Delta t \beta_{i j} \dxDG  f\left(u_h^{(j)}\right)\right), \quad \alpha_{ij},~\beta_{ij} \geq 0, \quad i=1, \cdots, s, \\
u_h^{n+1} &= u_h^{(s)}.
\end{align}		
\end{subequations}

\paragraph{cRKDG method.} For the cRKDG method \cite{chen2024runge}, let us introduce the local derivative operator $\dxl f(\cdot):\IV_h\to \IV_h$, such that
\begin{equation}
\int_K [\dxl f(u_h)]\cdot v_h  \dd x = - \int_K f(u_h) : \nabla v_h \dd x + \sum_{\ell=1}^{\nlK} \int_{e_{\ell,K}} \left(f(u_h^\inte)\cdot \nu_{\ell,K}\right)  \cdot  v_h \dd l \quad \forall v_h \in \IV_h.
\end{equation}
Note that, compared to $\dxDG f$, we use the interior trace instead of the numerical flux on the cell interfaces for $\dxl f$. The cRKDG scheme is given by replacing $\dxDG f$ in \eqref{eq-rkdg1} by $\dxl f$, namely, 
\begin{subequations}\label{eq:crkdg}
\begin{align}
u_h^{(i)} &= u_h^n -  \Delta t\sum_{j = 1}^{i-1}  a_{ij} \dxl f\left(u_h^{(j)}\right), \quad  i = 1, 2, \cdots, s,\label{eq-crkdg1}\\
u_h^{n+1} &= u_h^n - \Delta t \sum_{i = 1}^s b_i \dxDG  f\left(u_h^{(i)} \right).\label{eq:crkdg-2}
\end{align}
\end{subequations}
It is worth emphasizing that to preserve the conservative property, the cRKDG method must be constructed using the RK method in Butcher form \eqref{eq:rkdg}, rather than adopting the RK method in Shu--Osher form, as in \eqref{rk:ssp way1}.

\paragraph{Further notations.} For the sake of convenience in notation, let us define
\begin{equation}
\cF(u_h) = -\dxDG f(u_h) 
\quand \cG(u_h) = -\dxl f(u_h)
\end{equation}
for the DG operator and the local derivative operator, respectively. We denote by $\overline{v}_K:=\frac{1}{|K|}\int_K v(x)\dd x$ the cell average of $v$ on $K\in \setE_h$. Moreover, we define
\begin{equation}
\lambda = \max_{\ell,K}\Delta t\frac{|e_{\ell,K}|}{|K|} \quand \al = \max_{x,\ell,K}|\mathrm{eig}\left(\partial_u f(u)\cdot \nu_{\ell,K}\right)|.
\end{equation}
As a convention, we define $|e_{\ell,K}|=1$ if $d = 1$. In the rest of the paper, the symbol $c$ with appropriate subscripts is used to denote a generic positive constant independent of $h$ and $\Delta t$. 

\subsection{Bound-preserving SSP-RKDG method}
In this section, we provide a brief review of the bound-preserving SSP-RKDG method \cite{zhang2010maximum,zhang2010positivity,zhang2012maximum,qin2016bound,xing2010positivity,xing2013positivity,wu2023geometric}, which is constructed following the approach outlined below. 
\begin{itemize}[noitemsep]
\item[1.] Show that the solution of forward-Euler method preserves the weak bound-preserving property. In other words, under a time step constraint, the solution cell averages at the next time level remain in the invariant region. 
\item[2.] Apply a limiter to scale the solution polynomials towards the cell averages to preserve the pointwise bound for forward-Euler solutions.
\item[3.] Extend to high-order time stepping methods with the SSP-RK method. 
\end{itemize}

\subsubsection{Quadrature and convex decomposition of cell averages} In the design of the SSP-RKDG method, the integration along cell edges in \eqref{eq:dxDG} is approximated by a quadrature rule, i.e., the integral along the edge $e_{\ell,K}$ is computed by 
\begin{equation}
\int_{e_{\ell,K}} v \dd l \approx |e_{\ell,K}|\sum_{\beta = 1}^{{\rr}_\ell} \wB_{\beta,\ell,K} v(x_{\beta,\ell,K}), \qquad \forall \ell, K.
\end{equation}
Here, $\wB_{\beta,\ell,K}$ and $x_{\beta,\ell,K}$ ($\beta = 1, \cdots, {\rr}_\ell$) denote the associated quadrature weights and abscissae.
This quadrature rule should be exact for the integration of constants, i.e., it satisfies the Proposition~\ref{assp:edge}. 
\begin{PROP}\label{assp:edge}\it
For all cell $K$ and edge $\ell$, the quadrature rule on $e_{\ell,K}$ is exact for integrating constants. We have
\begin{equation}\label{eq:exact-edges}
\sum_{\beta = 1}^{{\rr}_\ell} \wB_{\beta,\ell,K} = 1. 
\end{equation}
\end{PROP}
A key ingredient in showing the weak bound-preserving property is to introduce an appropriate convex decomposition of the cell averages. See \cite{cui2023classic,cui2024optimal} and references therein. 
\begin{PROP}\label{assp:CAD}\it
There exists a convex decomposition of the cell average that includes all quadrature points at the cell interfaces $\{x_{\beta,\ell,K}\!:~ \beta = 1, \cdots, {\rr}_\ell,~ \ell = 1, \cdots, l(K)\}$. 
\begin{equation}\label{eq:cd}
\overline{v}_K = \sum_{\alpha = 1}^{{\rr}_K}\wI_{\alpha,K} v(x_{\alpha,K}) + \sum_{\ell=1}^{\nlK}\sum_{\beta = 1}^{{\rr}_\ell} \wI_{\beta,\ell,K} v(x_{\beta,\ell,K}), \quad 0<\wI_{\alpha,K}, \wI_{\beta,\ell,K}<1\quad \forall v\in \IZ^k(K).
\end{equation}
\end{PROP}
An illustration of the quadratures that enjoy Propositions~\ref{assp:edge} and \ref{assp:CAD} and are used in third-order and fourth-order schemes on 2D structured meshes is given in Figure~\ref{fig:quadrature_rule}.
Note that in particular, when taking $v(x) \equiv 1$, \eqref{eq:cd} in Proposition~\ref{assp:CAD} implies
\begin{equation}\label{eq:const1}
1 = \sum_{\alpha = 1}^{{\rr}_K}\wI_{\alpha,K} + \sum_{\ell=1}^{\nlK}\sum_{\beta = 1}^{{\rr}_\ell} \wI_{\beta,\ell,K}, \quad 0<\wI_{\alpha,K}, \wI_{\beta,\ell,K}<1.
\end{equation}
\begin{figure}[ht!]
\begin{center}
\begin{tabularx}{0.9\linewidth}{@{}c@{~~}C@{~}C@{~}C@{}}
\begin{sideways}{\hspace{0.5cm} third-order scheme}\end{sideways} &
\includegraphics[width=0.275\textwidth]{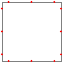} &
\includegraphics[width=0.275\textwidth]{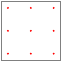}  &
\includegraphics[width=0.275\textwidth]{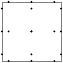} \\
\begin{sideways}{\hspace{0.5cm} fourth-order scheme}\end{sideways} &
\includegraphics[width=0.275\textwidth]{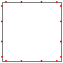} &
\includegraphics[width=0.275\textwidth]{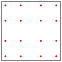}  &
\includegraphics[width=0.275\textwidth]{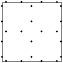} \\
\end{tabularx}
\caption{Quadratures in a square cell. From left to right: the quadrature points for edge integrals in Proposition~\ref{assp:edge} ($\{x_{\beta,\ell,K}\}$), for volume integrals, and for the convex decomposition in Proposition~\ref{assp:CAD} ($\{x_{\alpha,K},x_{\beta,\ell,K}\}$) for the weak bound-preserving property. The red points are constructed by Gauss quadrature and the black points are constructed by tensor product of Gauss quadrature and Gauss--Lobatto quadrature. The black points are used only in defining the bound-preserving limiter, and they are not used in calculating any numerical integration.}
\label{fig:quadrature_rule}
\end{center}
\end{figure}

In the rest of the paper, we denote by $S_K = \{x_{\alpha,K}, x_{\beta,\ell,K}\}$ the set of the points in the above convex decomposition on $K$. 
Note that $u_h$ is double-valued at cell interfaces, including all $\{x_{\beta,\ell,K}\}\subset S_K$. We will call $u_h \in B$ on  ${S}_K$ if  $u_h(x_{\alpha,K}) \in B$ and  $u_h^\inte(x_{\beta,\ell,K}) \in B$ for all $\alpha$ and $\beta$; we will call  
$u_h \in B$ on  $\overline{S}_K$ if $u_h \in B$ on ${S}_K$ and, in addition, $u_h^\exte(x_{\beta,\ell,K}) \in B$ for all $\beta,\ell$. We will use the short notations $u_{\alpha} = u_h(x_{\alpha})$ and $u_{\beta,\ell} = u_h^\inte(x_{\beta,\ell})$.

\subsubsection{Scalar hyperbolic conservation laws}

The solution to a scalar hyperbolic conservation law admits the maximum principle (e.g. \cite{dafermos2005hyperbolic}): if the initial condition is within the bound 
\begin{equation}\label{eq:scalar-bound}
    B = [m,M],
\end{equation}
then its solution will remain in the same region for all future time. 

The following proposition states that the forward-Euler method in conjunction with the DG spatial discretization preserves the weak maximum principle. See \cite{zhang2010maximum} for the 1D and 2D Cartesian meshes and \cite{zhang2012maximum} for the 2D triangular meshes. Further discussions can be found at \cite{cui2023classic,cui2024optimal}. 
\begin{PROP}[Weak maximum principle with $\cF$]\label{prop:weakp-scalar}\it
Consider scalar hyperbolic conservation laws in 1D or 2D over Cartesian or triangular meshes. Under the CFL condition 
\begin{equation}\label{CFL}
        \lambda \al \leq c_\mathrm{\cF},
\end{equation}
for $B$ defined in \eqref{eq:scalar-bound}, and for any $K\in \setE_h$, we have  
\begin{equation}
u_h \in B ~\text{on}~ \overline{S}_K \quad\Rightarrow\quad \overline{u_h + \Delta t \cF(u_h)}_K \in B.
\end{equation}
\revone{Here $c_\mathrm{\cF}$ is a positive constant. In the case that the same quadrature rule is used along all edges, $c_\mathrm{\cF}= \min_{\beta,\ell,K}\frac{\omega_{\beta,\ell,K}^\mathrm{I}}{\omega_{\beta,\ell,K}^\mathrm{B}}$ for Cartesian meshes and $c_\mathrm{\cF}= \min_{\beta,\ell,K}\frac{|e_{\ell,K}|}{\sum_{l=1}^{3}|e_{l,K}|}\frac{\omega_{\beta,\ell,K}^\mathrm{I}}{\omega_{\beta,\ell,K}^\mathrm{B}}$ for triangular meshes.}
\end{PROP}

After ensuring the weak maximum principle, a scaling limiter is applied to enforce the pointwise maximum principle of the solution $u_h$ on $S_K$.

\begin{PROP}\label{prop:limiter-scalar}\it
Consider scalar hyperbolic conservation laws in 1D or 2D. Let $p_K(x)$ be the solution polynomial on the mesh cell $K$, with its cell average $\overline{u}_{h,K}\in B = [m,M]$. Then define
\begin{equation}
\tilde{p}_K(x) = \theta \left(p_K(x) - \overline{u}_{h,K}\right) + \overline{u}_{h,K}, \quad 
\theta = \min\left\{\left|\frac{M-\overline{u}_{h,K}}{M_K-\overline{u}_{h,K}}\right|,\, \left|\frac{m-\overline{u}_{h,K}}{m_K-\overline{u}_{h,K}}\right|,\, 1\right\}
\end{equation}
with
\begin{equation}
M_K = \max_{x\in S_K} p_K(x) \quand m_K = \min_{x\in S_K} p_K(x).
\end{equation}
We have $p_K(x)\in B$, for all $x \in S_K$.
\end{PROP}

\begin{remark}\it
The scaling limiter in Proposition~\ref{prop:limiter-scalar} does not affect the spatial accuracy of the solution. The proof for the 1D case can be found in \cite{zhang2017positivity}. 
\end{remark}
From Proposition~\ref{prop:weakp-scalar} and Proposition~\ref{prop:limiter-scalar}, we know that under a sufficiently small time step size and with the scaling limiter, the forward-Euler solution to the scalar hyperbolic conservation laws preserves the maximum principle on $S_K$. This property can be extended to higher temporal order by using the SSP-RK methods. For example, for the third-order scheme, one can apply
\begin{subequations}
\begin{align}
u_h^{(1)} &= u_h^n + \Delta t \cF(u_h^n),\\
u_h^{(2)} &= \frac{3}{4} u_h^{n} + \frac{1}{4} \left(u_h^{(1)} + \Delta t \cF(u_h^{(1)})\right),\\ 
u_h^{n+1} &= \frac{1}{3} u_h^{n} + \frac{2}{3} \left(u_h^{(2)} + \Delta t \cF(u_h^{(2)})\right).
\end{align}
\end{subequations}
The scaling limiter is applied after each  RK stage if there exists a point in $S_K$ that violates the bounds. 

\subsubsection{Hyperbolic systems: compressible Euler equations}\label{sec:BP-system}
For systems of hyperbolic conservation laws, we primarily focus on the Euler equations for gas dynamics and utilize it as a concrete example for clarity of presentation. However, our analysis and numerical strategies can also be extended to general hyperbolic systems, including the shallow water equations and relativistic hydrodynamics.
\par
The compressible Euler equations take the form of \eqref{eq:claws} with 
\begin{equation}\label{eq:Euler}
    u = \begin{pmatrix}\rho\\
    m^\top\\
    E
    \end{pmatrix}\quand f(u) = \begin{pmatrix}
    m\\
    \frac{1}{\rho}m\otimes m + p I_d\\
    (E+p)\frac{m}{\rho}
    \end{pmatrix}. 
\end{equation}
Here 
\begin{equation}\label{eq:Euler-2}
m = \rho w, \quad 
E =\frac{|m|^2}{2\rho} + \revtwo{\rho} e, \quand 
p = (\gamma-1)\rho e. 
\end{equation}
For conservative variables, $\rho$ is the density, $m:\Omega \to \IR^d$ is the momentum, and $E$ is the total energy. In addition, $w:\Omega \to \IR^d$ is the velocity, $e$ is the internal energy, and $p$ is the pressure. The parameter $\gamma > 1$ is a constant ($\gamma = 1.4$ for air). 
The solution to the Euler equations preserves the invariant region 
\begin{equation}\label{eq:Euler-bound}
B = \left\{u = (\rho, m^\top, E)^\top\!: \rho > 0 \quand p = (\gamma - 1)\left(E - \frac{|m|^2}{2\rho}\right) > 0\right\}.
\end{equation}
Note that if \eqref{eq:Euler-bound} is violated, namely if the density $\rho$ or pressure $p$ becomes negative, then the system \eqref{eq:Euler} will be non-hyperbolic, and thus the initial value problem will be ill-posed.
For Euler equations, we have $\al = \max||w\cdot\nu|+c|$ and $c = \sqrt{\gamma p / \rho}$.  

It was proved that the forward-Euler DG scheme preserves the weak positivity of the Euler equations. See \cite{zhang2010positivity} for the 1D or 2D Cartesian meshes and \cite{zhang2012maximum} for the 2D triangular meshes. 

\begin{PROP}[Weak bound-preserving property with $\cF$]\label{prop:weakp-euler}\it
Consider Euler equations \eqref{eq:Euler} in 1D or 2D over Cartesian or triangular meshes. Under a specific CFL condition 
\begin{equation}\label{eq:CFL-Euler-F}
\lambda \al 
\leq c_\cF,
\end{equation}
for $B$ defined in \eqref{eq:Euler-bound}, and for any $K \in \setE_h$, we have
\begin{equation}
u_h \in B ~\text{on}~ \overline{S}_K \quad\Rightarrow\quad 
\overline{u_h + \Delta t \cF(u_h)}_K \in B.
\end{equation}
\revone{Here $c_\mathrm{\cF}$ is a positive constant. In the case that the same quadrature rule is used along all edges, $c_\mathrm{\cF}= \min_{\beta,\ell,K}\frac{\omega_{\beta,\ell,K}^\mathrm{I}}{\omega_{\beta,\ell,K}^\mathrm{B}}$ for Cartesian meshes and $c_\mathrm{\cF}= \min_{\beta,\ell,K}\frac{|e_{\ell,K}|}{\sum_{l=1}^{3}|e_{l,K}|}\frac{\omega_{\beta,\ell,K}^\mathrm{I}}{\omega_{\beta,\ell,K}^\mathrm{B}}$ for triangular meshes.}
\end{PROP}
As that for the scalar conservation laws, after showing weak bound-preserving property, we can apply a scaling limiter to preserve the pointwise bound over $S_K$. In below, we detail the scaling limiter for the Euler equations, which are also the hyperbolic systems we will test numerically in Section \ref{sec:numerical_experiment}. For those of the shallow water equations and the relativistic hydrodynamics, we refer to \cite{xing2010positivity,xing2013positivity} and \cite{qin2016bound}. 
\par
From Proposition~\ref{prop:weakp-euler}, we know that in each cell $K\in\setE_h$, the cell average of the DG solution $u_h = \transpose{(\rho_h, m_h^\top, E_h)}$ produced by the forward-Euler method belongs to the invariant region $B$.
Therefore, there exists a small positive number $\epsilon$, such that 
\begin{align}
\overline{\rho}_h \geq \epsilon \quad\text{and}\quad
\overline{p}_h = (\gamma - 1)\Big(\overline{E}_h - \frac{|\overline{m}_h|^2}{2\overline{\rho}_h}\Big) \geq \epsilon.
\end{align}
In practice, $\epsilon$ can be set to $10^{-8}$ or $10^{-13}$ in the computation \cite[Section~4.2]{zhang2017positivity}.
Define the numerical admissible set for the compressible Euler equations:
\begin{align}
B^\epsilon = \left\{u = \transpose{(\rho, m^\top, E)}\!:~ 
\rho \geq \epsilon \quand 
p = (\gamma-1)\Big(E - \frac{|m|^2}{2\rho}\Big) \geq \epsilon \right\}.
\end{align}
Now it suffices to apply a scaling limiter to modify the high-order moments of the DG solution, so that the scaled solution falls within $B^\epsilon$ and its cell averages are preserved. 
To this end, first, we enforce positivity of density by setting
\begin{subequations}\label{eq:Euler-limiter}
\begin{align}\label{eq:Euler-limiter1}
\widehat{\rho}_K = \theta_1 (\rho_K - \overline{\rho}_K) + \overline{\rho}_K,
\quad
\theta_1 = \min\biggl\{1,\, \frac{\overline{\rho}_K-\epsilon}{\overline{\rho}_K-\min\limits_{x_q\in S_K}\rho_K(x_q)}\biggr\},
\end{align}
where $\overline{\rho}_K$ denotes the cell average of density on cell $K$, namely $\rho_K = \rho_h|_K$. Hereinafter, we define $m_K$ and $E_K$ similarly.
Notice that $\widehat{\rho}_K$ and $\rho_K$ have the same cell average, and $\widehat{\rho}_K = {\rho}_K$ if $\min\limits_{x_q\in S_K}\rho_K(x_q)\geq \epsilon$.
Then, we define $\widehat{u}_h = \transpose{(\widehat{\rho}_h, \transpose{m_h}, E_h)}$ and enforce positivity of internal energy by setting
\begin{align}\label{eq:Euler-limiter2}
\widetilde{u}_K = \theta_2 (\widehat{u}_K - \overline{u}_K) + \overline{u}_K,
\quad
\theta_2 = \min\biggl\{1,\, \frac{\overline {\rho e}_K - \epsilon}{\overline{\rho e}_K - \min\limits_{x_q\in S_K}\rho e_K(x_q)}\biggr\},
\end{align}
\end{subequations}
where $\overline{\rho e}_K = \overline E_K - \frac{|\overline{m}_K|^2}{2\overline\rho_K}$
and $\rho e_K(x_q) = E_K(x_q) - \frac{|m_K(x_q)|^2}{2\rho_K(x_q)}$.
Note that $\widetilde{u}_K$ has the same cell average as $\widehat{u}_K$.
We refer to \cite{zhang2010positivity,zhang2017positivity,xu2017bound} for the justification of its high-order accuracy.  
\begin{PROP}\it
Consider the Euler equations in 1D or 2D. Let $u_{h,K}(x)$ be the solution polynomial on the mesh cell $K$, admitting the cell average of $\overline{u}_{h,K} \in B^\epsilon$. Then after applying the scalar limiters \eqref{eq:Euler-limiter1} and \eqref{eq:Euler-limiter2}, the resulting solution polynomial satisfies $\widetilde{u}_{h,K}\in B^\epsilon$.
\end{PROP}
The extension to the higher-order time stepping methods can be achieved via the SSP-RK methods similar to that has been discussed for scalar conservation laws. 

\section{Bound-preserving cRKDG method}\label{sec:bound_preserving}
The primary challenge in designing a bound-preserving cRKDG scheme is two-fold. 
First, the scheme involves not only the DG operator $\dxDG f$, but also the local derivative operator $\dxl f$. 
Second, the cRKDG schemes are based on RK methods in the Butcher form, as opposed to the Shu--Osher form.
As a result, the strategy used to prove bound preservation of SSP-RKDG methods cannot be directly applied to cRKDG schemes, i.e., our problem cannot be reduced to a single building block --- preserving the bound of $u_h + \Delta t \cF(u_h)$.   
\par
To address this challenge, we express the cRKDG scheme as convex combinations of forward-Euler steps, which involve not only $u_h + \Delta t \cF(u_h)$, but also $u_h + \Delta{t}\cG(u_h)$ and $u_h - \Delta{t}\cG(u_h)$. Our general route is outlined as follows. 
\begin{itemize}[noitemsep]
\item[1.] In addition to $u_h + \Delta t \cF(u_h)$, prove that $u_h+ \Delta t \cG(u_h)$ and $u_h- \Delta t \cG(u_h)$ also preserve the weak bound-preserving property. 
\item[2.] Write each stage of the RK method as a convex combination of the forward-Euler steps with $\cF$, $\cG$ and $-\cG$. Thus, each RK stage preserves the weak bound-preserving property.
\item[3.] After each RK stage, apply scaling limiters to preserve the bound on $S_K$. 
\end{itemize}
We emphasize that the convex decomposition of each RK stage is only employed to establish the weak bound-preserving property but not the implementation of the method. The scaling limiter will be applied after each complete RK stage, not after each forward-Euler step. 

\subsection{Forward-Euler steps with $\pm \cG$}
Let us define 
\begin{equation}
c_\cG = \min_{\beta,\ell,K}\frac{\wI_{\beta,\ell,K}}{\wB_{\beta,\ell,K}}.
\end{equation}

\subsubsection{Scalar conservation laws}

In Lemma \ref{prop:weakp-scalar-G}, we prove a counterpart to Proposition~\ref{prop:weakp-scalar} for the local derivative operator $\cG$.
\begin{LEM}[Weak maximum principle with $\pm \cG$]\label{prop:weakp-scalar-G}\it
Consider scalar conservation laws. If 
\begin{equation}\label{eq:CFL-G-scalar}
\lambda \al\leq c_\cG,
\end{equation}
then we have  
\begin{equation}
u_h \in B ~\text{on}~ S_K \quad\Rightarrow\quad 
\overline{ u_h \pm \Delta t \cG(u_h)}_K \in B.
\end{equation}
\end{LEM}
\begin{proof}
For ease of notation, we omit all subscripts $K$ when there is no confusion. Note that 
\begin{equation}\label{eq:Gbar}
\Delta t \overline{\cG(u_h)} = -\frac{\Delta t}{|K|}\sum_{\ell=1}^{\nlK}|e_{\ell}|\sum_{\beta = 1}^{{\rr}_\ell}\wB_{\beta,\ell}f(u_{\beta,\ell})\cdot \nu_{\ell} .
\end{equation}
Using the convex decomposition, we have
 \begin{equation}\label{eq:ubar}
        \overline{u}_h = \sum_{\alpha = 1}^{s}\wI_{\alpha} u_{\alpha} + \sum_{\ell=1}^{\nlK}\sum_{\beta = 1}^{{\rr}_\ell}\wI_{\beta,\ell} u_{\beta,\ell}, \qquad 0<\wI_\alpha, \wI_{\beta,\ell}<1.
    \end{equation}
    Therefore, 
\begin{equation}\label{eq:mp-H}
\begin{aligned}
    \overline{ u_h \pm \Delta t \cG(u_h)}   = \overline{u}_h \pm \Delta t \overline{\cG(u_h)} =\;& \sum_{\alpha = 1}^{{\rr}_K}\wI_{\alpha} u_{\alpha} +\sum_{\ell=1}^{\nlK}\sum_{\beta = 1}^{{\rr}_\ell} \wI_{\beta,\ell} \left(u_{\beta,\ell}\mp \Delta 
 t\frac{ |e_\ell|}{|K|}\frac{\wB_{\beta,\ell}}{ \wI_{\beta,\ell} }f(u_{\beta,\ell})\cdot \nu_{\ell}\right)\\
    :=\;& H(\{u_{\alpha}\},\{u_{\beta,\ell}\}),
    \end{aligned}
\end{equation}
Here $H(\{u_{\alpha}\},\{u_{\beta,\ell}\})$ means $H (\cdot)$ is a function depending on all nodal values of $u_h$ on $S_K$.

In below, we show the consistency and monotonicity of $H$.

First, given a constant $b$, we have the consistency $H(\{b\},\{b\})= b$. This property can be proved as follows. Using \eqref{eq:mp-H} and noting that $f(u_{\beta,\ell}) = f(b)$ is a constant, we have
\begin{align}
H(\{b\},\{b\}) 
= \overline{b} \pm \Delta t \overline{\cG(b)} 
= b \mp \frac{\Delta t}{|K|}f(b)\cdot\sum_{\ell = 1}^{l(K)} |e_{\ell}|\left(\sum_{\beta = 1}^{{\rr}_\ell} w_{\beta,\ell}^B\right) \nu_{\ell}.
\end{align}
Then we apply \eqref{eq:exact-edges} and note that $\nu_{\ell}$ is a constant on $e_\ell$. It yields
\begin{align}
 H(\{b\},\{b\}) 
 = b \mp \frac{\Delta t}{|K|}f(b)\cdot \sum_{\ell=1}^{l(K)} |e_{\ell}| \nu_{\ell} 
 = b \mp \frac{\Delta t}{|K|}f(b)\cdot \left(\sum_{\ell=1}^{l(K)} \int_{e_{\ell}}\nu_{\ell}\dd l\right) 
 = b.
\end{align}
For the last equality, we have used the identity $\sum_{\ell=1}^{l(K)} \int_{e_{\ell}}\nu_{\ell}\dd l  = 0$.  

Second, note that 
\begin{equation}
    \frac{\partial H}{\partial u_{\alpha}} = w_{\alpha}>0.
\end{equation}
With the time step constraint \eqref{eq:CFL-G-scalar}, we have
\begin{equation}
    \frac{\partial H}{\partial u_{\beta,\ell}} = \wI_{\beta,\ell} \mp  \Delta t \frac{|e_\ell|}{|K|}{\wB_{\beta,\ell}}\partial_u f(u_{\beta,\ell})\cdot \nu_{\beta,\ell}\geq \wI_{\beta,\ell} - \lambda \al \wB_{\beta,\ell}\geq 0.
\end{equation}
Hence $H$ is a monotonically increasing function with respect to all inputs.

Recall that $m\leq u_h \leq M$ on $S$, namely, $m\leq u_{\alpha},u_{\beta,\ell} \leq M$. Therefore, we can use the consistency and the monotonicity of $H$ to show that 
\begin{equation}
    m = H(\{m\},\{m\}) \leq H(\{u_{\alpha}\},\{u_{\beta,\ell}\}) \leq H(\{M\},\{M\})  = M. 
\end{equation}
Recall that we defined $\overline{ u_h \pm \Delta t \cG(u_h)} = H(\{u_{\alpha}\},\{u_{\beta,\ell}\})$. The proof of the lemma is completed. 
\end{proof}

\subsubsection{Hyperbolic systems: compressible Euler equations}

In Lemma \ref{prop:weakp-euler-G}, we establish a counterpart to Proposition~\ref{prop:weakp-euler} for the local derivative operator $\cG$. The proof of Lemma \ref{prop:weakp-euler-G} is based on the so-called Lax--Friedrichs splitting property, which is stated in Proposition~\ref{prop:system-pointwise}.  

\begin{PROP}[Lax--Friedrichs splitting property]\label{prop:system-pointwise}\it
For Euler equations in \eqref{eq:Euler}, with the invariant region $B$ defined in \eqref{eq:Euler-bound}, \revtwo{and for any unit vector $\nu$,} we have 
\begin{equation}
u \in B \quad\Rightarrow\quad u \pm \al^{-1}f(u)\cdot \nu \in B.
\end{equation}
\end{PROP}
\begin{proof}
We first show the positivity of density. From $u\in B$, we know $\rho >0$. The sound speed $c>0$ gives $\al =\max||\omega\cdot\nu|+c|>|\omega\cdot \nu|$. By looking into the first component of $u\mp \al^{-1} f(u)\cdot \nu$, we get 
\begin{equation}
\rho \mp \al^{-1}\rho w\cdot \nu = (1\mp \al^{-1}w\cdot \nu)\rho > 0.
\end{equation}
Next, let us show the positivity of pressure. From \eqref{eq:Euler}, we have
\begin{align}
u\mp \al^{-1} f(u)\cdot \nu = 
\begin{pmatrix}
(1\mp \al^{-1}w\cdot \nu)\rho \\
(1\mp \al^{-1} w\cdot \nu)m \mp \al^{-1} p\nu \\
(1\mp \al^{-1} w\cdot \nu)E \mp \al^{-1} w\cdot \nu p
\end{pmatrix}.
\end{align}
Recall that $p = (\gamma - 1)(E-\frac{1}{2}\frac{|m|^2}{\rho})$. 
Let the bracket $[\cdot]$ below denote the dependence.
By performing a direct calculation similar to the one in \cite[Remark 2.4]{zhang2010positivity}, we have
\begin{equation}\label{eq:compLF}
\begin{aligned}
&\frac{1}{\gamma-1}p[u\mp \al^{-1} f(u)\cdot \nu] \\
=&\, (1\mp \al^{-1} w\cdot \nu)E \mp \al^{-1} w\cdot \nu p - \frac{1}{2} \frac{|(1\mp \al^{-1} w\cdot \nu)m \mp \al^{-1} p\nu|^2}{(1\mp \al^{-1}w\cdot \nu)\rho}\\
=&\, (1\mp \al^{-1} w\cdot \nu)\Big(E - \frac{1}{2}\frac{|m|^2}{\rho}\Big) 
\mp \Big(\al^{-1} w\cdot\nu p - \al^{-1}m\cdot\nu\frac{p}{\rho}\Big) 
- \frac{1}{2}\frac{p^2}{\al^2(1\mp \al^{-1} w\cdot \nu)\rho}\\
=&\, (1\mp \al^{-1} w\cdot \nu)\left(\frac{p}{\gamma-1} - \frac{1}{2(\al\mp w\cdot \nu)^2}\frac{p^2}{\rho}\right)\mp \Big(\al^{-1} w\cdot\nu p - \al^{-1}m\cdot\nu\frac{p}{\rho}\Big).
\end{aligned}
\end{equation}
Recall that $w$ is defined by $m/\rho$. Thus, the last term in \eqref{eq:compLF} vanishes, which gives
\begin{align}
p[u\mp \al^{-1} f(u)\cdot \nu] = \left(1 - \frac{\gamma-1}{2(\al\mp w\cdot\nu)^2}\frac{p}{\rho}\right)(1\mp \al^{-1} w\cdot\nu)p. 
\end{align}
Notice that $1\mp \al^{-1} w\cdot \nu>0$. We have 
\begin{align}
p[u\mp \al^{-1} f(u)\cdot \nu] > 0 ~~&\Leftrightarrow~~
\frac{(\gamma-1)}{2(\al\mp  w\cdot \nu)^2}\frac{p}{\rho} < 1 \nonumber\\ ~~&\Leftrightarrow~~ 
\gamma\frac{p}{\rho}<\frac{2\gamma}{\gamma - 1}(\al \mp w\cdot \nu)^2 ~~\Leftrightarrow~~ \sqrt{\gamma\frac{p}{\rho}} < \sqrt{\frac{2\gamma}{\gamma - 1}}(\al \mp w\cdot \nu).
\end{align}
The sound speed $c = \sqrt{\gamma p/\rho}$ implies $\al \geq ||w\cdot \nu|+c|$. Therefore, from the constant $\gamma>1$, we have $\sqrt{\frac{2\gamma}{\gamma-1}} = \sqrt{2+\frac{2}{\gamma-1}}>1$, which gives $ p[u\mp \al^{-1} f(u)] > 0$.
\end{proof}
\revone{
\begin{remark}
A simple alternative proof of Proposition~\ref{prop:system-pointwise} can be obtained using the geometric quasilinearization. We refer to Appendix~\ref{app:GQL} for details.     
\end{remark}
}
\begin{LEM}[Weak bound-preserving property with  $\pm\cG$]\label{prop:weakp-euler-G}\it
Suppose that Proposition~\ref{prop:system-pointwise} holds for a hyperbolic system. If
\begin{equation}\label{eq:CFL-Euler-G}
\lambda \al \leq c_\cG,
\end{equation}
then we have  
\begin{equation}
u_h \in B ~\text{on}~ S_K \quad\Rightarrow\quad
\overline{ u_h \pm \Delta t \cG(u_h)}_K \in B.
\end{equation}
\end{LEM}
\begin{proof}
    Similar to the proof in Lemma \ref{prop:weakp-scalar-G}, we omit all subscripts $K$ when there is no confusion. With a similar computation, we can acquire \eqref{eq:Gbar} and \eqref{eq:ubar} for the hyperbolic system. It gives
  \begin{equation}\label{eq:G-Euler-convex}
\begin{aligned}
    \overline{ u_h \pm \Delta t \cG(u_h)}
    &= \sum_{\alpha = 1}^{{\rr}_K}\wI_{\alpha} u_{\alpha} +\sum_{\ell=1}^{\nlK}\sum_{\beta = 1}^{{\rr}_\ell} \left(\wI_{\beta,\ell}-\Delta 
 t\frac{ |e_\ell|}{|K|} \al{\wB_{\beta,\ell}}\right)u_{\beta,\ell}\\
    &+\sum_{\ell=1}^{\nlK}\sum_{\beta = 1}^{{\rr}_\ell} \left(\Delta 
 t\frac{ |e_\ell|}{|K|}\al{\wB_{\beta,\ell}}\right)\left(u_{\beta,\ell}\mp \al^{-1}f(u_{\beta,\ell})\cdot \nu_{\ell}\right).
\end{aligned}
\end{equation}
Invoking Proposition~\ref{prop:system-pointwise}, we have $u_{\beta,\ell}\mp {\al^{-1}}f(u_{\beta,\ell})\cdot \nu_{\ell} \in B$. Under the CFL condition \eqref{eq:CFL-Euler-G},  $\wI_{\alpha}$, $\wI_{\beta,\ell}-\Delta 
 t\frac{|e_\ell|}{|K|} a_0{\wB_{\beta,\ell}}$, and $\Delta 
 t\frac{|e_\ell|}{|K|}\al{\wB_{\beta,\ell}}$ are nonnegative numbers in $[0,1]$ summing to $1$. Therefore, \eqref{eq:G-Euler-convex} indicates that $\overline{ u_h \pm \Delta t \cG(u_h)}_K$ is a convex combination of  $u_{\alpha},u_{\beta,\ell}, u_{\beta,\ell}\mp {\al^{-1}}f(u_{\beta,\ell})\cdot \nu_{\ell}\in B$. Since $B$ is a convex region, we have $\overline{ u_h \pm \Delta t \cG(u_h)}\in B$. 
\end{proof}

\subsection{Bound-preserving framework for cRKDG schemes}
In this section, we consider discretization of scalar hyperbolic conservation laws with the invariant region $B = [m,M]$ and hyperbolic systems with the invariant regions $B$ defined accordingly. When referring to the scaling limiters that preserve bounds over $S_K$, we refer to the limiter in Proposition~\ref{prop:limiter-scalar} for scalar conservation laws and the limiters in \eqref{eq:Euler-limiter1} and \eqref{eq:Euler-limiter2} for Euler equations. For other hyperbolic systems, we refer to the limiters in \cite{xing2010positivity,xing2013positivity} for shallow water equations and  \cite{qin2016bound} for relativistic hydrodynamics. Moreover, we assume Propositions \ref{assp:edge} and \ref{assp:CAD} are valid.

\begin{LEM}\label{lem:bp-stage}\it
Consider the $i$th stage of a cRKDG method $u_h^{(i)}$. Suppose \revthr{that the cell average of} $u_h^{(i)}$ can be written as a convex combination of \revthr{those of} forward-Euler steps with $\cF$ and $\pm \cG$ for previous stages. In other words, it can be written as a convex combintation of 
\begin{equation}\label{eq:FEs}
\revthr{\overline{u_h^{(j)} + \Delta t \mu_{ij}^\cF \cF\left(u_h^{(j)}\right)}} \quand \revthr{\overline{u_h^{(j)} \pm \Delta t \mu_{ij}^{\pm\cG} \cG\left(u_h^{(j)}\right)}}\quad \text{with} \quad j<i, \quad \mu_{ij}^\cF, \mu_{ij}^{\pm\cG}>0.
\end{equation} 
Consequently, if
\begin{equation}\label{eq:CFL}
\lambda \al
\leq \min_{j<i}\left(\frac{c_\cF}{\mu_{ij}^\cF},\frac{c_\cG}{{\mu_{ij}^{\pm\cG}}}\right),
\end{equation}
then we have 
\begin{equation}
u_h^{(j)} \in B ~\text{on}~ \overline{S}_K ~\text{for all}~ j<i
\quad\Rightarrow\quad 
{\overline{u}}_{h,K}^{(i)}\in B  ~\text{on}~ S_K.
\end{equation}
\end{LEM}
\begin{proof}
Note that under the given time step constraint \eqref{eq:CFL}, from Propositions~\ref{prop:weakp-scalar}, \ref{prop:weakp-scalar-G}, \ref{prop:weakp-euler}, and \ref{prop:weakp-euler-G}, we have 
\begin{equation}
\overline{u_h^{(j)} + \Delta t \mu_{ij}^\cF \cF(u_h^{(j)})}_K \in B \quand \overline{u_h^{(j)} \pm \Delta t \mu_{ij}^\mathrm{\pm\cG} \cG(u_h^{(j)})}_K\in B.
\end{equation}
As a result, using the fact that $\overline{\sum_{i = 1}^j c_i v_i}_K = \sum_{i = 1}^jc_i \overline{v_i}_K$, one can see that the cell average of the convex combinations of the forward-Euler steps \eqref{eq:FEs} will also be in the region $B$. Thus, $u_h^{(i)}$ satisfies the weak bound-preserving property. 
\end{proof}
\revthr{With Lemma~\ref{lem:bp-stage}, we can prove the main result of the paper, Theorem~\ref{thm:bpcrkdg}, whose detailed proof is provided in Appendix~\ref{app:proofmain}.}
\begin{THM}[Bound-preserving cRKDG method]\label{thm:bpcrkdg}\it
Consider a cRKDG method. Suppose \revthr{that, without using limiters,} each of its stages can be written as a convex \revtwo{combination} of forward-Euler steps involving $\cF$ and $\pm \cG$ for previous stages. Under the CFL condition 
\begin{equation}
    \lambda \al \leq \min_{i,j} \left( \frac{c_\cF}{\mu_{ij}^\cF},\frac{c_\cG}{\mu_{ij}^{\pm\cG}}\right),
\end{equation}
if we apply the corresponding scaling limiter after each RK stage, then we have 
\begin{equation}
u_h^n \in B ~\text{on}~ S_K ~\text{for all}~ K 
\quad\Rightarrow\quad 
u_h^{n+1} \in B ~\text{on}~ S_K ~\text{for all}~ K.
\end{equation}
\end{THM}
\revone{
\begin{remark}[CFL condition for stability] In the analysis of this paper, the time step constraint is imposed solely to ensure the weak bound-preserving property of the cRKDG method. In practice, however, the time step must also satisfy the stability requirement. The CFL numbers for the $(k+1)$th-order cRKDG methods with $\IP^k$ elements, where $k = 1, 2, 3$, are 0.333, 0.178, and 0.103, respectively \cite{chen2024runge}. These values are obtained via a Fourier-type analysis. For comparison, the corresponding CFL numbers ensuring stability for the RKDG methods are 0.333, 0.209, and 0.145, respectively \cite{cockburn2001runge}.
\end{remark}}

\revthr{
\begin{remark}[Applying limiter after each stage]
We clarify that, the limiter is applied after each cRKDG stage, not after each forward-Euler step. The decomposition of an RK stage into a convex combination of forward-Euler steps is used solely to prove the weak bound-preserving property of each stage and is not required in the implementation of the solver.
\end{remark}}
\begin{remark}[Compactness]\it
    Since the scaling limiter requires only local information on each mesh cell, the proposed bound-preserving cRKDG method shares the same stencil as the original cRKDG method and preserves its compactness.
\end{remark}
\begin{remark}[Extension to other hyperbolic systems]\it\label{rmk:other_systems}
Although we have been focusing on the Euler equations, the idea can be naturally extended to other hyperbolic systems, such as the shallow water equations and the relativistic hydrodynamics. As can be seen, the Lax--Friedrichs splitting property is central to the weak bound-preserving property of $\pm\cG$. For a different hyperbolic system, if we can prove similar propositions as those in Propositions~\ref{prop:weakp-euler} and \ref{prop:system-pointwise}, then we can design the bound-preserving cRKDG schemes for the new system. 
\par
For shallow water equations, the weak positivity with $\cF$ has been proved in \cite{xing2010positivity,xing2013positivity} over 1D or 2D Cartesian or triangular meshes; the Lax--Friedrichs splitting property can be proved by following similar lines as those in Proposition~\ref{prop:system-pointwise}.
\par
For relativistic hydrodynamics, the weak bound-preserving property for $\cF$ has been proved in \cite{qin2016bound} for Cartesian meshes; the proof of the Lax--Friedrichs splitting property can be found at \cite[Lemma 2.3]{wu2015high} and also in \cite{qin2016bound}. 
\end{remark}
\begin{remark}[Extension to compressible Navier--Stokes equations]\it
The same discussion also applies to compressible Navier--Stokes equations with positivity-preserving numerical flux in \cite{zhang2017positivity}. The design of the cRKDG method for convection-diffusion equations with applications to the Navier--Stokes equations along with its positivity-preserving strategy will be discussed in future work.  
\end{remark}

We end this section by mentioning that the fully discrete cRKDG scheme in Theorem~\ref{thm:bpcrkdg} preserves the local conservation in the sense that
\begin{equation}
\overline{u}_{h,K}^{n+1} = \overline{u}_{h,K}^{n} -\frac{\Delta t}{|K|}\sum_{\ell = 1}^{l(K)} \widetilde{\int}_{e_{\ell,K}} \left(\sum_{i = 1}^sb_i \widehat{f\cdot\nu_{\ell,K}}^{(i)}\right)\dd l. 
\end{equation}
Here, $\widetilde{\int}\cdot \dd l$ is the numerical quadrature in Proposition~\ref{assp:edge} for approximating the edge integral, and $\widehat{f\cdot\nu_{\ell,K}}^{(i)}$ is the numerical flux computed with the $i$th stage solution $u_h^{(i)}$. As a result, the cRKDG scheme also preserves the global conservation 
\begin{align}
\int_\Omega u_h^{n+1} \dd x = \int_\Omega u_h^{n} \dd x.
\end{align}
\begin{remark}[Extension to higher dimensions]\it
Although we primarily focus on 1D and 2D cases, the results in this paper can be easily extended to higher dimensions. 
\end{remark}

\subsection{Choices of RK methods}

In this section, we construct RK methods in Butcher form, where the stages are expressed as convex combinations of the forward-Euler steps \eqref{eq:FEs}. We determine RK coefficients that result in less restrictive CFL constraints as prescribed by Theorem \ref{thm:bpcrkdg}. Our focus is exclusively on RK methods with nonnegative coefficients. Hence, $\alpha$, $\beta$, $a$, and $b$, possibly with subscripts, all represent nonnegative numbers in this section. For simplicity, we will omit the subscript $h$. All decimal numbers are rounded to four decimal places.

Note that \( c_\cF \) and \( c_\cG \) may vary depending on the settings. To avoid unnecessary complications and to design RK methods suitable for general cases, we do not distinguish between operators \(\cF\) and \(\pm\cG\) and instead minimize the largest possible \(\mu_{ij}^\cF\) and \(\mu_{ij}^{\pm \cG}\).
This parameter search is similar in flavor to the maximization of SSP coefficients in the design of optimal SSP-RK methods \cite{gottlieb2011strong}. 

We note that despite efforts to maximize the time step size in Theorem \ref{thm:bpcrkdg}, which provides a more relaxed theoretical constraint for bound preservation, the corresponding \(\Delta t\) will not be used in our numerical experiments. Instead, an adaptive time-marching strategy, as outlined in Section~\ref{sec:timestepsize}, is employed.  

\subsubsection{The second-order case}
The cRKDG scheme with the generic second-order RK time integrator takes the form:
\begin{subequations}\label{eq:RK2}
\begin{align}
u^{(2)} &= u^n + a_{21}\Delta t \cG(u^n),\label{eq:rk2-1}\\
u^{n+1} &= u^n + \Delta t \left(b_1 \cF(u^n) + b_2 \cF(u^{(2)})\right).\label{eq:rk2-2}
\end{align}
\end{subequations}
Here the parameters $a_{21}$, $b_1$ and $b_2$ are RK coefficients. 

Our first task is to write \eqref{eq:RK2} as a convex combination of \eqref{eq:FEs}. Note that the first-stage \eqref{eq:rk2-1} is already in the form of a forward-Euler step. We only need to decompose \eqref{eq:rk2-2}. We add and subtract the term $\alpha u^{(2)}$, with $\alpha$ being an underdetermined coefficient. By invoking the definition of $u^{(2)}$ in \eqref{eq:rk2-1}, we have 
\begin{equation}\label{eq:rk2-decomp}
\begin{aligned}
u^{n+1} &= u^n -\alpha u^{(2)} + \Delta t b_1 \cF(u^n) + \alpha u^{(2)} + \Delta t b_2 \cF(u^{(2)})\\
&= \left(\left(1-\alpha\right)u^n - \Delta t \alpha a_{21}  \cG(u^n) +  \Delta t b_1 \cF(u^n)\right) + \alpha u^{(2)} + \Delta t b_2 \cF(u^{(2)})\\
&= \left(1-\alpha\right)\left(u^n - \Delta t \frac{\alpha a_{21}}{1-\alpha}  \cG(u^n) +  \Delta t \frac{b_1}{1-\alpha} \cF(u^n)\right) + \alpha \left(u^{(2)} + \Delta t \frac{b_2}{\alpha} \cF(u^{(2)})\right).
\end{aligned}
\end{equation}
Further splitting $u^n$ gives
\begin{equation}\label{eq:rk2decomp}
\begin{aligned}
	u^{n+1} = &\left(1-\alpha\right)\beta\left(u^n -  \frac{\Delta t \alpha a_{21}}{(1-\alpha)\beta}  \cG(u^n)\right) + \left(1-\alpha\right)(1-\beta)\left(u^n + \frac{\Delta t  b_1}{(1-\alpha)(1-\beta)} \cF(u^n)\right)\\
 &+ \alpha \left(u^{(2)} + \frac{\Delta t b_2}{\alpha} \cF(u^{(2)})\right).
 \end{aligned}
 \end{equation}
Now, with the parameters $0\leq \alpha,\beta\leq 1$, $u^{n+1}$ has also been written as a convex combination of the forward-Euler steps \eqref{eq:FEs}.

Then, we minimize the coefficients $\mu_{ij}^\cF$ and $\mu_{ij}^{\pm\cG}$. \revone{There are three types of constraints that must be considered. First, \eqref{eq:rk2decomp} must be a convex combination of forward-Euler steps, which requires $0 \leq \alpha, \beta \leq 1$. Second, the forward-Euler step with the DG operator $\cF$ must advance the solution forward in time; in addition to the previous condition, this requires $0 \leq b_1, b_2$. Finally, to achieve second-order temporal accuracy, the RK method must satisfy the order conditions $b_1 + b_2 = 1$ and $b_2 a_{21} = \frac{1}{2}$, which can be derived by matching the coefficients in the Taylor expansion of the discrete solution with those of the continuous solution. For further details, we refer to \cite[Section II]{hairer1993solving}.}

\revone{With the aforementioned constraints, we have the following optimization problem:}
\begin{subequations}\label{eq:opt-2-0}
\begin{align}	&\min_{\alpha,\beta,a_{21},b_1,b_2} \max\left\{a_{21},\frac{\alpha a_{21}}{(1-\alpha)\beta},\frac{b_{1}}{(1-\alpha)(1-\beta)},\frac{b_2}{\alpha}\right\}\label{eq:opt-obj-2}\\
  &\text{subject to} \quad  0\leq \alpha,\beta \leq 1, \quad 0\leq b_1,b_2,\quad b_1 + b_2 = 1, \quad b_2a_{21} = \frac{1}{2}\label{eq:opt-cstr-2}.
\end{align}
\end{subequations}
Here, the function value becomes infinity when its denominator vanishes. With the observation 
\begin{equation}
\min_{0\leq \beta\leq 1} \max\left(\frac{a}{\beta}, \frac{b}{1-\beta}\right) = a+b,
\end{equation} we can get rid of $\beta$ and reduce the optimization problem \eqref{eq:opt-obj-2} and \eqref{eq:opt-cstr-2} as 
\begin{subequations}\label{eq:opt-2}
\begin{align}
	&\min_{\alpha,a_{21},b_1,b_2} \max\left\{a_{21},\frac{\alpha a_{21}+b_1}{1-\alpha},\frac{b_2}{\alpha}\right\}\label{eq:opt-obj-2-1}
\\
  &\text{subject to}\quad  0\leq \alpha \leq 1, \quad 0\leq b_1,b_2, \quad b_1 + b_2 = 1, \quad b_2a_{21} = \frac{1}{2}. 
\end{align}    
\end{subequations}
We use the $\texttt{fminimax}$ function in MATLAB to solve the minimax constraint problem. The optimal solution is non-unique. A simple one we found is 
\begin{equation}
	\alpha = \sqrt{3} -1, \quad a_{21} = \frac{1}{2},\quad b_1 = 0, \quad b_2 = 1.  
\end{equation}
The objective function \eqref{eq:opt-obj-2-1} is approximately $1.3660$. This results in the explicit midpoint rule.
\noindent
\textbf{The second-order method:}
\begin{subequations}\label{eq:cRKDG2}
\begin{align}
u^{(2)} &= u^n + \frac{\Delta t }{2} \cG(u^n),\label{eq:0-rk2-1}\\
u^{n+1} &= u^n + \Delta t  \cF(u^{(2)}).\label{eq:0-rk2-2}
\end{align}
\end{subequations}
In the bound-preserving cRKDG method, a scaling limiter is applied twice --- after \eqref{eq:0-rk2-1} and \eqref{eq:0-rk2-2}, respectively --- to enforce the pointwise bounds of the numerical solution. The method is typically coupled with $\IP^1$- or $\IQ^1$-DG method to achieve overall second-order accuracy.
\begin{remark}\it
Note in general, we have the decomposition 
\begin{equation}
u + \Delta t \mu^{\cF} \cF(u) \pm \Delta t \mu^{\pm \cG} \cG(u) = \beta\left(u + \Delta t \frac{\mu^\cF}\beta \cF(u) \right) +\left(1-\beta\right)\left( u +  \Delta t \frac{\mu^{\pm \cG}}{1-\beta} \cG(u)\right).
\end{equation}
When minimizing the maximum coefficients, we have 
\begin{equation}
\min_{0\leq \beta\leq 1} \max\left\{\frac{\mu^{\cF}}{\beta}, \frac{\mu^{\pm \cG}}{1-\beta}\right\} = \mu^{\cF}+\mu^{\pm \cG}.
\end{equation} 
As a result, when dealing with a forward-Euler step with both $\cF$ and $\pm\cG$, we can skip further decompositions and take the combined value $\mu^{\cF}+\mu^{\pm \cG}$ as the effective coefficient. This fact will be used in the construction of the third- and fourth-order schemes. 
\end{remark}
\revtwo{
\begin{remark}
    The key to the parameter search of the cRKDG method lies in expressing it in a generalized Shu--Osher form, as in \eqref{eq:rk2-1} and \eqref{eq:rk2decomp}. We consider it ``generalized" because we distinguish the forward-Euler steps that share the same stage $u^{(i)}$ but involve different spatial operators. 
    Compared with the standard SSP-RKDG method, the parameter search of the cRKDG method has the following constraint and flexibility. 
    \begin{enumerate}
        \item (Constraint)  When written in Butcher form, the final stage of the cRKDG method must not contain the operator $\mathcal{G}$. This restriction is equivalent to imposing additional constraints on the generalized Shu--Osher coefficients.
        \item (Flexibility) When expressed in the generalized Shu--Osher form, the cRKDG method allows the coefficients in front of $\cG$ to be negative.  
    \end{enumerate}
    The constraint tends to reduce the SSP coefficients of the cRKDG method relative to those of the standard SSP-RKDG method, resulting in a more restrictive time-step condition. On the other hand, the flexibility enables the method to overcome certain order barriers, as will be demonstrated in the fourth-order case. This flexibility is akin to allowing downwind stages in the standard SSP-RKDG scheme, which can also overcome the order barrier but may introduce complications when imposing downwind flux at the outflow boundary. However, thanks to the compactness of $\cG$, the cRKDG method does not face this issue.  
\end{remark}
}

\subsubsection{The third-order case}
Now we repeat the procedure to obtain a third-order RK scheme. The generic third-order RK scheme takes the form
\begin{subequations}
\begin{align}
	u^{(2)} &= u^n + \Delta t a_{21} \cG(u^n),\\
	u^{(3)} &= u^n + \Delta t \left(a_{31} \cG(u^n) + a_{32} \cG(u^{(2)})\right),\\
	u^{n+1} &= u^n + \Delta t \left(b_1 \cF(u^n)+b_2 \cF(u^{(2)}) + b_3 \cF(u^{(3)})\right).
\end{align}
\end{subequations}
$u^{(2)}$ has already been in the form of a forward-Euler step. The decomposition of $u^{(3)}$ can be done by following similar lines as in \eqref{eq:rk2-decomp}. 
\begin{equation}\label{eq:rk3-decomp-1}
\begin{aligned}
	u^{(3)} &= u^n -\alpha_1 u^{(2)} + \Delta t a_{31} \cG(u^n) + \alpha_1 u^{(2)} + \Delta t a_{32} \cG(u^{(2)})\\
			&= \left(1-\alpha_1\right)u^n-\Delta t \alpha_1 a_{21} \cG(u^n) + \Delta t a_{31} \cG(u^n) + \alpha_1 u^{(2)} + \Delta t a_{32}\cG(u^{(2)})\\
			&= \left(1-\alpha_1\right)\left(u^n + \Delta t\left(\frac{a_{31}-\alpha_1a_{21}}{1-\alpha_1}\right)\cG(u^n)\right) + \alpha_1\left(u^{(2)} + \Delta t \frac{a_{32}}{\alpha_1}\cG(u^{(2)})\right).
\end{aligned}
\end{equation}
The decomposition of $u^{n+1}$ is written as 
\begin{equation}\label{eq:rk3-decomp-2}
	\begin{aligned}
	u^{n+1} &= u^n + \Delta t \left(b_1 \cF(u^n)+b_2 \cF(u^{(2)}) + b_3 \cF(u^{(3)})\right)\\
			&= u^n -\alpha_2 u^{(2)} -\alpha_3 u^{(3)} + \Delta t b_1 \cF(u^n) + \alpha_2 u^{(2)} + \Delta t b_2 \cF(u^{(2)}) + \alpha_3 u^{(3)} + \Delta t b_3 \cF(u^{(3)})\\
			&= \left(1-\alpha_2-\alpha_3\right) u^n - \Delta t (\alpha_2 a_{21} + \alpha_3a_{31})\cG(u^n) - \Delta t \alpha_3a_{32} \cG(u^{(2)})\\
			&+ \Delta t b_1 \cF(u^n) + \alpha_2 u^{(2)} + \Delta t b_2 \cF(u^{(2)}) + \alpha_3 u^{(3)} + \Delta t b_3 \cF(u^{(3)})\\
			&= \left(1-\alpha_2-\alpha_3\right)\left(u^n + \Delta t \left(-\frac{\alpha_2a_{21}+\alpha_3a_{31}}{1-\alpha_2-\alpha_3}\cG(u^n)+\frac{b_1}{1-\alpha_2-\alpha_3}\cF(u^n)\right)\right)\\
			&+\alpha_2\left(u^{(2)}+\Delta t\left(-\frac{\alpha_3a_{32}}{\alpha_2}\cG(u^{(2)}) + \frac{b_2}{\alpha_2}\cF(u^{(2)})\right)\right) + \alpha_3\left(u^{(3)}+\Delta t \frac{b_3}{\alpha_3}\cF(u^{(3)})\right).
	\end{aligned}
\end{equation}
To obtain a third-order RK scheme, we need to solve the optimization problem
\begin{equation}
\min_{\alpha_1,\alpha_2,\alpha_3,a_{21},a_{31},a_{32},b_1,b_2,b_3}\max \left\{a_{21},\frac{|a_{31}-\alpha_1a_{21}|}{1-\alpha_1},\frac{a_{32}}{\alpha_1},\frac{\alpha_2a_{21}+\alpha_3a_{31}+b_1}{1-\alpha_2-\alpha_3} ,\frac{\alpha_3a_{32}+b_2}{\alpha_2},\frac{b_3}{\alpha_3}\right\}
\end{equation}
together with the order conditions and the constraints on the domain of the parameters. The optimizer is in general not unique. One simple optimal solution we found in the region
\begin{equation}
(0,0,0,0,0,0,0,0,0)\leq (\alpha_1,\alpha_2,\alpha_3,a_{21},a_{31},a_{32},b_1,b_2,b_3)\leq (1,1,1,10,10,10,1,1,1)
\end{equation}
is given by
\begin{equation}
\begin{aligned}
(\alpha_1,\alpha_2,\alpha_3,a_{21},a_{31},a_{32},b_1,b_2,b_3)
=  (0.4764, 0.2442, 0.5242,1/3,0,2/3,1/4,0,3/4).
\end{aligned}
\end{equation}
The corresponding value for the objective function is $1.4308$. It gives the Heun's third-order method. 

\noindent
\textbf{The third-order method:}
\begin{subequations}\label{eq:cRKDG3}
\begin{align}
	u^{(2)} &= u^n +  \frac{\Delta t}{3} \cG(u^n),\label{eq:cRKDG3-1}\\
	u^{(3)} &= u^n + \Delta t \frac{2}{3} \cG(u^{(2)}),\label{eq:cRKDG3-2}\\
	u^{n+1} &= u^n + \Delta t \left(\frac{1}{4} \cF(u^n)+\frac{3}{4} \cF(u^{(3)}) \right).\label{eq:cRKDG3-3}
\end{align}
\end{subequations}
In the bound-preserving cRKDG method, a scaling limiter is applied three times --- after \eqref{eq:cRKDG3-1}, \eqref{eq:cRKDG3-2}, and \eqref{eq:cRKDG3-3}, respectively --- to enforce the pointwise bounds of the numerical solution. The method is typically coupled with $\IP^2$- or $\IQ^2$-DG method to achieve overall third-order accuracy.

\subsubsection{The fourth-order case}
As before, we repeat a similar procedure to acquire a fourth-order RK scheme. The generic fourth-order RK scheme takes the following form. 
\begin{subequations}
\begin{align}
u^{(2)} &= u^n + \Delta t a_{21} \cG(u^n),\\
u^{(3)} &= u^n + \Delta t \left(a_{31} \cG(u^n) + a_{32} \cG(u^{(2)})\right),\\
u^{(4)} &= u^n + \Delta t \left(a_{41} \cG(u^n) + a_{42} \cG(u^{(2)})+a_{43}\cG(u^{(3)}\right),\\
u^{n+1} &= u^n + \Delta t \left(b_1 \cF(u^n)+b_2 \cF(u^{(2)}) + b_3 \cF(u^{(3)}) + b_4 \cF(u^{(4)}\right).
\end{align}
\end{subequations}

$u^{(2)}$ has already been in the form of a forward-Euler step. The decomposition of $u^{(3)}$ is the same as that in \eqref{eq:rk2-decomp} and \eqref{eq:rk3-decomp-1}.
\begin{equation}
\begin{aligned}
u^{(3)} 
= \left(1-\alpha_1\right)\left(u^n + \Delta t\left(\frac{a_{31}-\alpha_1a_{21}}{1-\alpha_1}\right)\cG(u^n)\right) + \alpha_1\left(u^{(2)} + \Delta t \frac{a_{32}}{\alpha_1}\cG(u^{(2)})\right).
\end{aligned}
\end{equation}
The decomposition of $u^{(3)}$ is the same as that in \eqref{eq:rk3-decomp-2}.
\begin{equation}
\begin{aligned}
    u^{(4)} 
    &= (1-\alpha_2-\alpha_3)\left(u^n+\Delta t\frac{a_{41}-\alpha_2 a_{21}-\alpha_3 a_{31}}{1-\alpha_2-\alpha_3}\cG(u^n)\right) \\
    & + \alpha_2 \left(u^{(2)} + \Delta t \frac{a_{42} -\alpha_3 a_{32}}{\alpha_2}\cG(u^{(2)})\right) + \alpha_3 \left(u^{(3)} + \Delta t \frac{a_{43}}{\alpha_3} \cG(u^{(3)})\right).
\end{aligned}
\end{equation}
Finally, $u^{n+1}$ can be decomposed as the following. 
\begin{equation}
\begin{aligned}
u^{n+1} &= u^n -\alpha_4 u^{(2)} - \alpha_5 u^{(3)} - \alpha_6 u^{(4)} + \Delta t b_1 \cF(u^n) \\
&+ \alpha_4u^{(2)} +  \Delta t b_2 \cF(u^{(2)}) + \alpha_5 u^{(3)} + \Delta t b_3 \cF(u^{(3)}) + \alpha_6 u^{(4)} + \Delta t \cF(u^{(4)}) \\
&= (1-\alpha_4-\alpha_5-\alpha_6)u^n - \Delta t (\alpha_4a_{21} + \alpha_5a_{31} + \alpha_6a_{41}) \cG(u^n) + \Delta t b_1 \cF(u^n) \\
&+ \alpha_4 u^{(2)} - \Delta t (\alpha_5 a_{32} + \alpha_6 a_{42})\cG(u^{(2)}) + b_2 \cF(u^{(2)})\\
&+\alpha_5u^{(3)}- \Delta t\alpha_6a_{43}\cG(u^{(3)}) + \Delta t b_3 \cF(u^{(3)}) +\alpha_6 u^{(4)} + \Delta t b_4 \cF(u^{(4)}).
\end{aligned}
\end{equation}
To find a fourth-order RK method with less restrictive time step constraint, we need to solve the optimization problem 
\begin{equation}
\begin{aligned}
    \min_{\alpha_i,a_{ij},b_i} \max\left\{a_{21},\frac{|a_{31}-\alpha_1a_{21}|}{1-\alpha_1},\frac{a_{32}}{\alpha_1},\frac{|a_{41}-(\alpha_2 a_{21}+\alpha_3 a_{31})|}{1-\alpha_2-\alpha_3},\frac{|a_{42}-\alpha_3a_{32}|}{\alpha_2},\frac{a_{43}}{\alpha_3},\right.\\
    \left.\frac{\alpha_4a_{21}+\alpha_5a_{31}+\alpha_6a_{41}+b_1}{1-\alpha_4-\alpha_5-\alpha_6},
\frac{\alpha_5a_{32}+\alpha_6a_{42}+b_2}{\alpha_4},\frac{\alpha_6a_{43}+b_3}{\alpha_5},\frac{b_4}{\alpha_6}\right\}
    \end{aligned}
\end{equation}
subject to the order conditions and the constraints on the searching domain. Finding the global minimum is difficult in general. One feasible solution with small objective function is 
\begin{equation}
\begin{aligned}&(\alpha_1,\alpha_2,\alpha_3,\alpha_4,\alpha_5,\alpha_6,a_{21},a_{31},a_{32},a_{41},a_{42},a_{43},b_1,b_2,b_3,b_4)\\
=\,& (0.5000, 0.2346, 0.6850,    0.3334, 0.3066, 0.1142, 1/2, 0,1/2,0,0,1,1/6,1/3,1/3,1/6).
\end{aligned}
\end{equation}
The value of corresponding objective function is $1.4598$. It gives the classic fourth-order RK method. 

\noindent
\textbf{The fourth-order method:}
\begin{subequations}\label{eq:cRKDG4}
\begin{align}
u^{(2)} &= u^n + \frac{\Delta t}{2} \cG(u^n),\label{eq:cRKDG4-1}\\
u^{(3)} &= u^n + \frac{\Delta t}{2} \cG(u^{(2)}),\label{eq:cRKDG4-2}\\
u^{(4)} &= u^n + \Delta{t}\, \cG(u^{(3)}),\label{eq:cRKDG4-3}\\
u^{n+1} &= u^n + \Delta{t} \left(\frac{1}{6} \cF(u^n)+\frac{1}{3} \cF(u^{(2)}) +\frac{1}{3} \cF(u^{(3)}) +\frac{1}{6} \cF(u^{(4)}) \right).\label{eq:cRKDG4-4}
\end{align}
\end{subequations}
In the bound-preserving cRKDG method, a scaling limiter is applied four times --- after \eqref{eq:cRKDG4-1}, \eqref{eq:cRKDG4-2}, \eqref{eq:cRKDG4-3}, and \eqref{eq:cRKDG4-4}, respectively --- to enforce the pointwise bounds of the numerical solution. The method is typically coupled with $\IP^3$- or $\IQ^3$-DG method to achieve overall fourth-order accuracy.

\section{Numerical experiments}\label{sec:numerical_experiment}

In this section, we validate the accuracy of our numerical schemes by convergence rate tests. Representative physical benchmarks for the linear advection equation and compressible Euler equations indicate that the cRKDG schemes preserve bounds/positivity and are stable for demanding gas dynamic simulations.  Cartesian meshes are used for the 2D tests.  
\par
It is important to note that we only apply the bound-/positivity-preserving limiter without any special treatment of anti-oscillation. This is to specifically emphasize that this paper focuses on bound preservation, namely, our goal is to minimize the use of other limiters to showcase the bound-preserving property of the cRKDG methods.  
In addition, for all tests related to the compressible Euler equations, the ideal gas equation of state with $\gamma = 1.4$ is used. And $\epsilon = 10^{-8}$ is used for the numerical admissible set $B^{\epsilon}$.

\subsection{Implementation and time-marching strategy}\label{sec:timestepsize}
We provide details on implementing the simulator with an adaptive time step size strategy. 
Recall that the linear stability CFL of cRKDG methods are: $0.333$ for the second-order scheme; $0.178$ for the third-order scheme; and $0.103$ for the fourth-order scheme \cite{chen2024runge}.
At each time step, a tentative time step size is give by
\begin{align}\label{eq:numerical_experiment:CFL}
\Delta t = \frac{\mathrm{CFL}}{a_0} \Delta x.
\end{align}
If an out-of-bound value is produced in the solution, then we decrease the time step size by a factor of two and recompute. 
\par
We adopt the same adaptive time-stepping used in \cite{wang2012robust, zhang2017positivity}.
As an example, we list our adaptive time-stepping algorithm for the third-order cRKDG scheme of solving the compressible Euler system. A similar time-marching strategy can be constructed for scalar conservation laws.  
\begin{itemize}[leftmargin=0.5cm]
\item[] {\bf Algorithm}. At time $t^n$, use linear stability CFL to compute the trial step size $\Delta{t}$ by \eqref{eq:numerical_experiment:CFL}. The parameter $\epsilon > 0$ is a prescribed small number for numerical admissible state set $B^\epsilon$.
The input DG polynomial $u_h^n$ satisfies $u_h^n(x_q)\in B^\epsilon$, for all $x_q\in \cup_K S_K$. 
\item[] Step~1. Given DG polynomial $u_h^n$, compute the first stage to obtain $u_h^{(2)}$.
\begin{itemize}[topsep=0pt, leftmargin=1.0cm, label=\labelitemi]
\item If the cell averages $\overline{u}_K^{(2)} \in B^\epsilon$, for all $K\in\setE_h$, then apply limiter \eqref{eq:Euler-limiter} to obtain $\widetilde{u}_h^{(2)}$ and go to Step~2.
\item Otherwise, recompute the first stage with halved step size $\Delta{t} \leftarrow \frac{1}{2}\Delta{t}$. 
\end{itemize}
\item[] Step~2. Given DG polynomial $\widetilde{u}_h^{(2)}$, compute the second stage to obtain $u_h^{(3)}$.
\begin{itemize}[topsep=0pt, leftmargin=1.0cm, label=\labelitemi]
\item If the cell averages $\overline{u}_K^{(3)} \in B^\epsilon$, for all $K\in\setE_h$, then apply limiter \eqref{eq:Euler-limiter} to obtain $\widetilde{u}_h^{(3)}$ and go to Step~3.
\item Otherwise, return to Step~1 and restart the computation with halved step size $\Delta{t} \leftarrow \frac{1}{2}\Delta{t}$. 
\end{itemize}
\item[] Step~3. Given DG polynomial $\widetilde{u}_h^{(3)}$, compute the third stage to obtain $u_h^{(4)}$.
\begin{itemize}[topsep=0pt, leftmargin=1.0cm, label=\labelitemi]
\item If the cell averages $\overline{u}_K^{(4)} \in B^\epsilon$, for all $K\in\setE_h$, then apply limiter \eqref{eq:Euler-limiter} to obtain $u_h^{n+1}$. We finish the current cRKDG step.
\item Otherwise, return to Step~1 and restart the computation with halved step size $\Delta{t} \leftarrow \frac{1}{2}\Delta{t}$. 
\end{itemize}
\end{itemize}
A time-marching algorithm, as defined above, can easily lead to an endless loop if the scheme loses positivity. However, the results proved in Section~\ref{sec:bound_preserving} ensure that the loop terminates within a finite number of restarts.

\subsection{Convergence study}
In this part, we utilize manufactured solutions to verify the accuracy of the cRKDG schemes. 
We consider two representative examples: a linear scalar conservation law, specifically the linear advection equation in one dimension, and a nonlinear system of conservation laws, namely the compressible Euler equations in two dimensions. 
\par
For the tested cRKDG algorithms, we combine $\IP^k$- or $\IQ^k$-DG schemes for $k = 1, 2, 3$ with $(k+1)$th-order RK time discretization, as specified in \eqref{eq:cRKDG2}, \eqref{eq:cRKDG3}, and \eqref{eq:cRKDG4}, respectively. Let $\texttt{err}_{\Delta x}$ denote the error on a mesh with resolution $\Delta x$. The convergence rate is computed by $\ln(\texttt{err}_{\Delta x}/\texttt{err}_{\Delta x/2})/\ln(2)$.

\paragraph{Example~4.1 (1D linear advection).}
Let us consider the equation $\partial_t u + \partial_x u = 0$ in the computational domain $\Omega = [0,1]$ with a simulation end time $T = 0.1024$. The prescribed non-polynomial solution is $u = \sin{2\pi (x-t)} + 2$, and $u\in[1,3]$. The initial and boundary conditions are defined by the manufactured solution.
\par
We employ $k$th-order ($k=1, 2, 3$) Lagrange polynomials as basis functions for $\IP^k$ spaces and use $(k+1)$-point Gauss quadrature for numerical integration. 
Select time step sizes small enough such that the CFL condition is always satisfied.
We obtained the optimal convergence rates, e.g., $(k+1)$th-order in $\IP^k$ space, see Table~\ref{tab:convergence_lin_adv}.
The bound-preserving limiter is triggered.
\begin{table}[ht!]
\centering
{\small
\begin{tabularx}{0.95\linewidth}{@{~}c@{~}|c@{~}C@{~}c@{~}|c@{~}C@{~}c@{~}|c@{~}C@{~}c@{~}}
\toprule
~ & ~ & $k = 1$ & ~ & ~ & $k = 2$ & ~ & ~ & $k = 3$ & ~ \\
\toprule
$\Delta t$ & $\Delta x$ & {$\|u_h^{N_T}-u(T)\|_{L^2}$} & rate & $\Delta x$ & {$\|u_h^{N_T}-u(T)\|_{L^2}$} & rate & $\Delta x$ & {$\|u_h^{N_T}-u(T)\|_{L^2}$} & rate \\
\midrule
$2^{4}\cdot10^{-4}$ & $1/2^5$ &  $4.444\cdot10^{-3}$ & ---    & $1/2^4$ & $4.641\cdot10^{-4}$ & ---    & $1/2^3$ & $2.390\cdot10^{-4}$ & ---    \\
$2^{3}\cdot10^{-4}$ & $1/2^6$ & $1.167\cdot10^{-3}$ & 1.929  & $1/2^5$ & $5.822\cdot10^{-5}$ & 2.995  & $1/2^4$ & $1.213\cdot10^{-5}$ & 4.301  \\
$2^{2}\cdot10^{-4}$ & $1/2^7$ & $3.109\cdot10^{-4}$ & 1.908  & $1/2^6$ & $7.272\cdot10^{-6}$ & 3.001  & $1/2^5$ & $6.553\cdot10^{-7}$ & 4.211  \\
$2^{1}\cdot10^{-4}$ & $1/2^8$ & $8.166\cdot10^{-5}$ & 1.929  & $1/2^7$ & $9.090\cdot10^{-7}$ & 3.000  & $1/2^6$ & $4.013\cdot10^{-8}$ & 4.030  \\ 
\bottomrule
\end{tabularx}
}
\caption{Test of accuracy (linear advection). From left to right: the errors and convergence rates for $\IP^1$, $\IP^2$, and $\IP^3$ schemes.}
\label{tab:convergence_lin_adv}
\end{table}

\paragraph{Example~4.2 (2D compressible Euler).}
Let us consider the Euler equations in the computational domain $\Omega = [0,1]^2$ with simulation end time $T = 0.1024$. 
We use the tensor product of $k$th-order ($k=1, 2, 3$) Legendre polynomials as bases for $\IP^k$ spaces and the tensor product of $k$th-order Lagrange polynomials as bases for $\IQ^k$ spaces.  We use the tensor product of $(k+1)$-point Gauss quadrature for numerical integration. 
The discrete $L^2$ distance between the numerical solution and the exact solution is defined by
\begin{align}
\|u_h^{N_T} - u(T)\|_{L^2}^2 = \|\rho_h^{N_T}-\rho(T)\|_{L^2}^2 + \|m_h^{N_T}-m(T)\|_{L^2}^2 + \|E_h^{N_T}-E(T)\|_{L^2}^2.
\end{align}
\par
We consider two prescribed solutions as follows. In the first scenario, the solution is non-polynomial for all unknowns to better validate the accuracy, and it produces a nonzero right-hand side source. In the second scenario, we employ manufactured solutions with a density close to zero.
\begin{align}
\text{(scenario~I)}&\quad
\begin{cases}
\rho\, = \exp{(-t)} \sin{2\pi (x+y)} + 2,\\
w = \begin{pmatrix}
\exp{(-t)} \cos{(2\pi x)} \sin{(2\pi y)} + 2\\
\exp{(-t)} \sin{(2\pi x)} \cos{(2\pi y)} + 2
\end{pmatrix},\\
e\, = \frac{1}{2}\exp{(-t)} \cos{2\pi (x+y)} + 1.
\end{cases}\\
\text{(scenario~II)}&\quad
\begin{cases}
\rho\, = 0.99\sin{\pi(x + y - 2t)} + 1,\\
w = \begin{pmatrix}
1 \\ 1
\end{pmatrix},\\
e\, = \displaystyle \frac{1}{(\gamma - 1)}\frac{1}{0.99\sin{\pi(x + y - 2t)} + 1}.
\end{cases}
\end{align}
The right-hand side source and initial and boundary conditions are defined by above manufactured solutions.
Select time step sizes small enough such that the CFL condition is always satisfied.
We obtained the optimal convergence rates, e.g., $(k+1)$th-order in both $\IP^k$ and $\IQ^k$ spaces; see Table~\ref{tab:convergence_Euler_test1} and Table~\ref{tab:convergence_Euler_test2}.
\begin{table}[ht!]
\centering
\begin{tabularx}{0.9\linewidth}{@{~}c@{~}|c@{~}|c@{~}|C@{~}C@{~}|C@{~}C@{~}}
\toprule
\multicolumn{3}{c|}{~} & \multicolumn{2}{c|}{$\IP^k$ basis} & \multicolumn{2}{c}{$\IQ^k$ basis}  \\
\toprule
$k$ & $\Delta t$ & $\Delta x$ & $\|u_h^{N_T} - u(T)\|_{L^2}$ & rate & $\|u_h^{N_T} - u(T)\|_{L^2}$ & rate \\
\midrule
$1$ & $2^{5}\cdot10^{-5}$ & $1/2^4$ & $2.631\cdot10^{-2}$ & ---   & $2.573\cdot10^{-2}$ & ---   \\
~   & $2^{4}\cdot10^{-5}$ & $1/2^5$ & $6.293\cdot10^{-3}$ & 2.064 & $6.855\cdot10^{-3}$ & 1.908 \\
~   & $2^{3}\cdot10^{-5}$ & $1/2^6$ & $1.561\cdot10^{-3}$ & 2.011 & $1.781\cdot10^{-3}$ & 1.945 \\
~   & $2^{2}\cdot10^{-5}$ & $1/2^7$ & $3.912\cdot10^{-4}$ & 1.997 & $4.546\cdot10^{-4}$ & 1.970 \\
\midrule
$2$ & $2^{5}\cdot10^{-5}$ & $1/2^3$ & $1.103\cdot10^{-2}$ & ---   & $4.642\cdot10^{-3}$ & ---   \\
~   & $2^{4}\cdot10^{-5}$ & $1/2^4$ & $1.337\cdot10^{-3}$ & 3.045 & $5.587\cdot10^{-4}$ & 3.055 \\
~   & $2^{3}\cdot10^{-5}$ & $1/2^5$ & $1.647\cdot10^{-4}$ & 3.021 & $7.109\cdot10^{-5}$ & 2.975 \\
~   & $2^{2}\cdot10^{-5}$ & $1/2^6$ & $2.057\cdot10^{-5}$ & 3.001 & $9.040\cdot10^{-6}$ & 2.975 \\
\midrule
$3$ & $2^{5}\cdot10^{-5}$ & $1/2^2$ & $1.600\cdot10^{-2}$ & ---   & $3.373\cdot10^{-3}$ & ---   \\
~   & $2^{4}\cdot10^{-5}$ & $1/2^3$ & $8.637\cdot10^{-4}$ & 4.145 & $1.803\cdot10^{-4}$ & 4.226 \\
~   & $2^{3}\cdot10^{-5}$ & $1/2^4$ & $5.192\cdot10^{-5}$ & 4.030 & $1.086\cdot10^{-5}$ & 4.054 \\
~   & $2^{2}\cdot10^{-5}$ & $1/2^5$ & $3.232\cdot10^{-6}$ & 4.002 & $6.704\cdot10^{-7}$ & 4.018 \\
\bottomrule
\end{tabularx}
\caption{Test of accuracy (Euler, scenario~I). The errors and convergence rates for $\IP^k$ and $\IQ^k$ ($k=1,2,3$) spaces.}
\label{tab:convergence_Euler_test1}
\end{table}
\begin{table}[ht!]
\centering
\begin{tabularx}{0.9\linewidth}{@{~}c@{~}|c@{~}|c@{~}|C@{~}C@{~}|C@{~}C@{~}}
\toprule
\multicolumn{3}{c|}{~} & \multicolumn{2}{c|}{$\IP^k$ basis} & \multicolumn{2}{c}{$\IQ^k$ basis}  \\
\toprule
$k$ & $\Delta t$ & $\Delta x$ & $\|u_h^{N_T} - u(T)\|_{L^2}$ & rate & $\|u_h^{N_T} - u(T)\|_{L^2}$ & rate \\
\midrule
$1$ & $2^{5}\cdot10^{-5}$ & $1/2^4$ & $7.317\cdot10^{-3}$ & ---   & $1.358\cdot10^{-2}$ & ---   \\
~   & $2^{4}\cdot10^{-5}$ & $1/2^5$ & $1.680\cdot10^{-3}$ & 2.123 & $3.507\cdot10^{-3}$ & 1.953 \\
~   & $2^{3}\cdot10^{-5}$ & $1/2^6$ & $3.806\cdot10^{-4}$ & 2.142 & $8.770\cdot10^{-4}$ & 2.000 \\
~   & $2^{2}\cdot10^{-5}$ & $1/2^7$ & $8.484\cdot10^{-5}$ & 2.166 & $2.198\cdot10^{-4}$ & 1.997 \\
\midrule
$2$ & $2^{5}\cdot10^{-5}$ & $1/2^3$ & $1.805\cdot10^{-3}$ & ---   & $3.271\cdot10^{-3}$ & ---   \\
~   & $2^{4}\cdot10^{-5}$ & $1/2^4$ & $2.204\cdot10^{-4}$ & 3.034 & $5.056\cdot10^{-4}$ & 2.694 \\
~   & $2^{3}\cdot10^{-5}$ & $1/2^5$ & $2.750\cdot10^{-5}$ & 3.003 & $6.848\cdot10^{-5}$ & 2.884 \\
~   & $2^{2}\cdot10^{-5}$ & $1/2^6$ & $3.440\cdot10^{-6}$ & 2.999 & $8.842\cdot10^{-6}$ & 2.953 \\
\midrule
$3$ & $2^{5}\cdot10^{-5}$ & $1/2^2$ & $1.736\cdot10^{-3}$ & ---   & $8.852\cdot10^{-4}$ & ---   \\
~   & $2^{4}\cdot10^{-5}$ & $1/2^3$ & $9.500\cdot10^{-5}$ & 4.192 & $5.721\cdot10^{-5}$ & 3.952 \\
~   & $2^{3}\cdot10^{-5}$ & $1/2^4$ & $5.740\cdot10^{-6}$ & 4.049 & $3.314\cdot10^{-6}$ & 4.110 \\
~   & $2^{2}\cdot10^{-5}$ & $1/2^5$ & $3.997\cdot10^{-7}$ & 3.844 & $1.804\cdot10^{-7}$ & 4.200 \\
\bottomrule
\end{tabularx}
\caption{Test of accuracy (Euler, scenario~II). The errors and convergence rates for $\IP^k$ and $\IQ^k$ ($k=1,2,3$) spaces.}
\label{tab:convergence_Euler_test2}
\end{table}

\subsection{Physical benchmark tests in one dimension}
In this part, we validate the cRKDG schemes using three representative 1D benchmarks, including the traveling of Heaviside function governed by the linear advection equation for scalar conservation law, as well as the Lax shock tube and double rarefraction governed by compressible Euler equations for nonlinear system.
\par 
For all tests, we utilize $k$th-order ($k=2, 3$) Lagrange polynomials as basis functions for $\IP^k$ spaces and use $(k+1)$-point Gauss quadrature for numerical integration. 
The $(k+1)$th-order RK method is combined with the $\IP^k$ space discretization. 

\paragraph{Example~4.3 (Traveling Heaviside function).}
Let us consider the linear advection equation $\partial_t u + \partial_x u = 0$. We choose the computational domain $\Omega = [0,1]$ and set the simulation end time as $T = 1$. 
The initial condition is specified as $u^0(x) = 1$, and we use Dirichlet boundary condition $u = 2$ on left end of the domain. Hence $u\in [1,2]$.
We uniformly partition the domain $\Omega$ into $200$ cells.
Our schemes preserve bounds and conservation. The discontinuity is correctly captured. Figure~\ref{fig:lin_adv_1D} displays snapshots of the numerical solution $u_h$ produced by $\IP^2$ and $\IP^3$ schemes at $T = 0.5$. 
Compared to the case where no bound-preserving limiter is applied (see the top of Figure~\ref{fig:lin_adv_1D}), the solution stays in the bounds after applying the limiter (see the bottom of Figure~\ref{fig:lin_adv_1D}). 
The exact solution is plotted in red as a reference.
\begin{figure}[ht!]
\begin{center}
\begin{tabularx}{0.95\linewidth}{@{}c@{\quad}c@{\quad}c@{}}
\begin{sideways}{$\hspace{0.95cm} \text{without limiter} \quad$}\end{sideways} &
\includegraphics[width=0.45\textwidth]{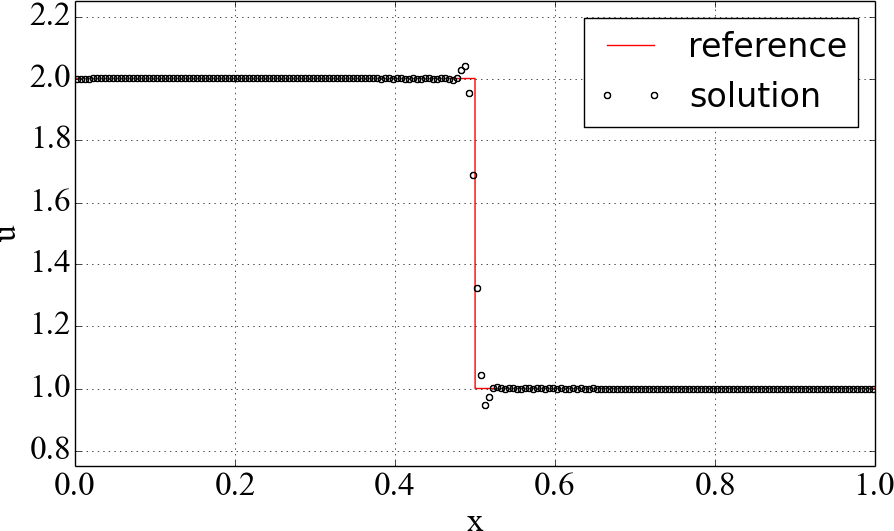} &
\includegraphics[width=0.45\textwidth]{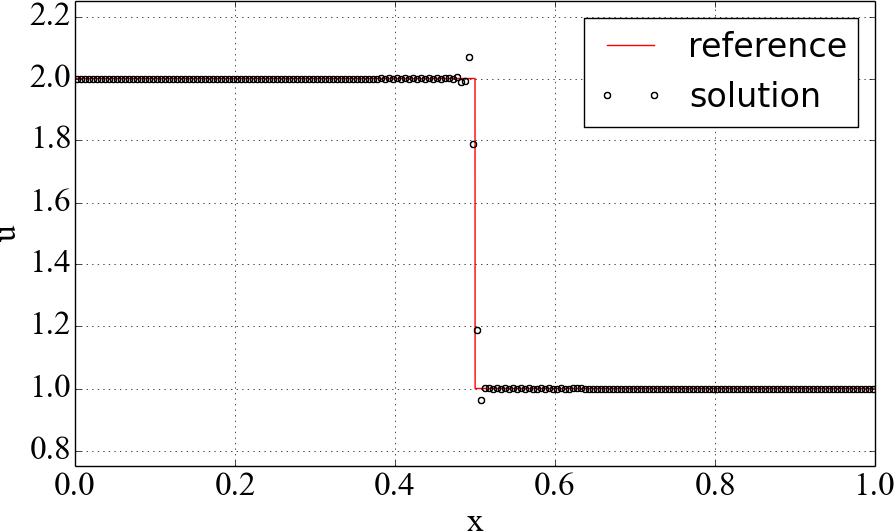} \\
\begin{sideways}{$\hspace{1.5cm} \text{with limiter} \quad$}\end{sideways} &
\includegraphics[width=0.45\textwidth]{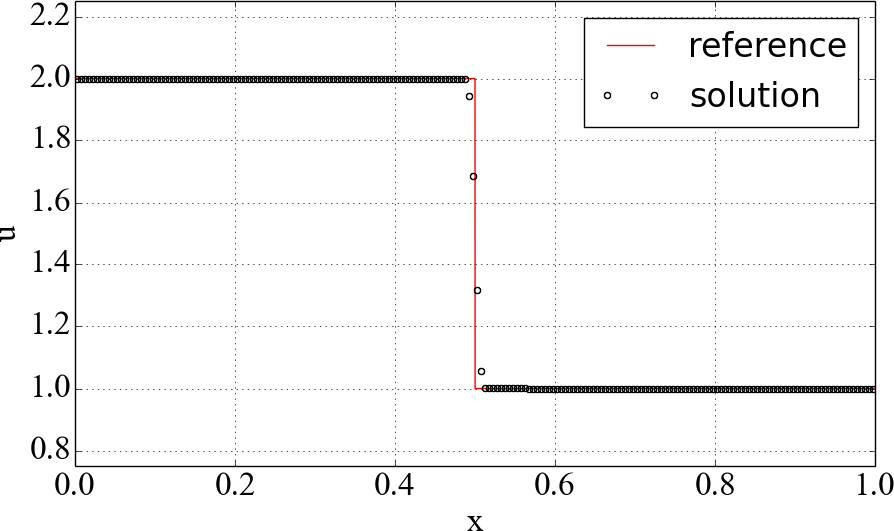} &
\includegraphics[width=0.45\textwidth]{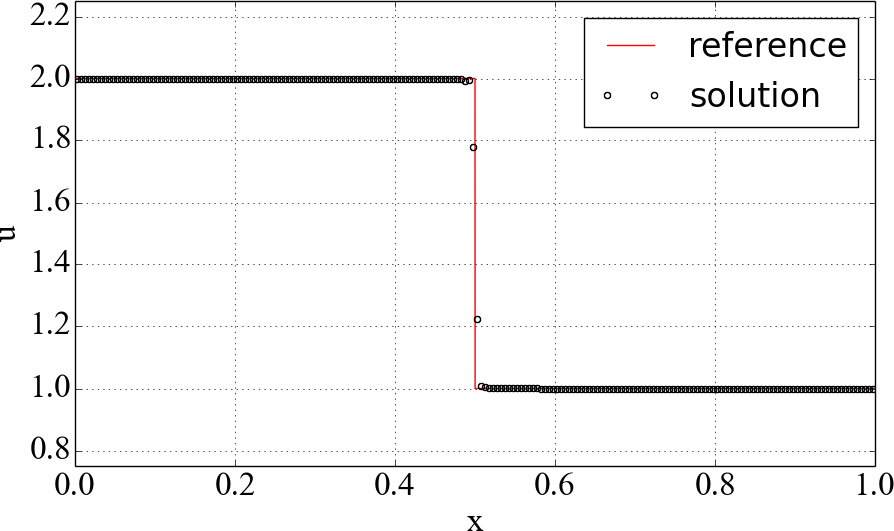} \\
~ &
~~$\IP^2$ scheme & 
~~$\IP^3$ scheme \\
\end{tabularx}
\caption{1D traveling Heaviside function. Snapshots are taken at $T = 0.5$ on $200$ uniform cells. Top: simulations without applying bound-preserving limiter, which produce solutions with out-of-bound cell averages. Bottom: simulations with bound-preserving limiter and solutions stay within the desired bounds. Only cell averages are plotted.}
\label{fig:lin_adv_1D}
\end{center}
\end{figure}

\paragraph{Example~4.4 (Lax shock tube).}
The Lax shock tube problem serves as a classical benchmark for testing numerical schemes of compressible Euler equations.
We choose the computational domain $\Omega = [-5,5]$ and set the simulation end time as $T = 1.3$. 
The initial condition is given by 
\begin{align}
\transpose{(\rho^0, u^0, p^0)} = \begin{cases}
\transpose{(0.445, 0.698, 3.528)}, &\text{if}~x < 0, \\
\transpose{(0.5, 0, 0.571)}, &\text{if}~x \geq 0.
\end{cases} 
\end{align}
We supplement Dirichlet boundary conditions $\transpose{(\rho, m, E)} = \transpose{(0.445, 0.698, 3.528)}$ on the left end of $\Omega$ and $\transpose{(\rho, m, E)} = \transpose{(0.5, 0, 0.571)}$ on the right end of $\Omega$.
\par
We uniformly partition the domain $\Omega$ into $200$ cells.
Our schemes preserve positivity and conservation. The shock location is correctly captured. Figure~\ref{fig:Euler_lax_1D} shows snapshots of the density profiles produced by $\IP^2$ and $\IP^3$ schemes at time $T=1.3$. 
\begin{figure}[ht!]
\begin{center}
\begin{tabularx}{0.95\linewidth}{@{}c@{\quad}c@{}}
\includegraphics[width=0.45\textwidth]{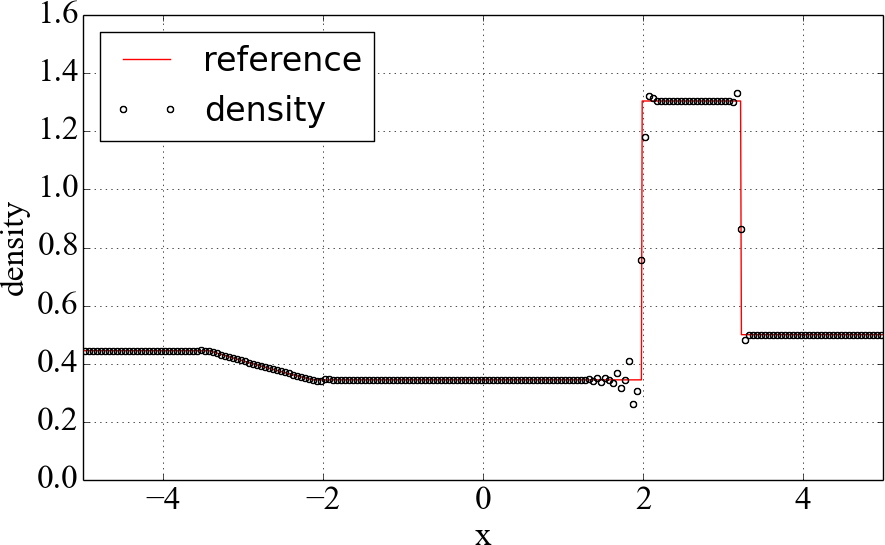} &
\includegraphics[width=0.45\textwidth]{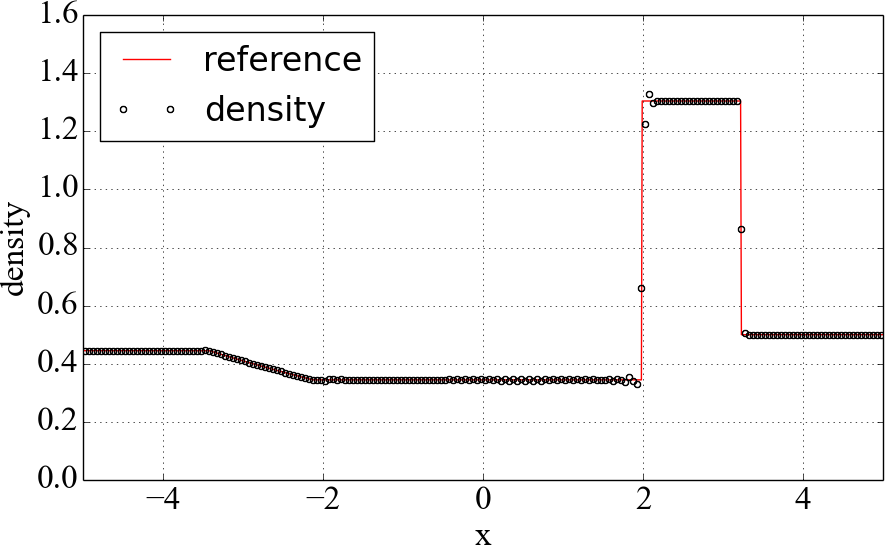} \\
$\qquad~ \IP^2$ scheme & 
$\qquad~ \IP^3$ scheme \\
\end{tabularx}
\caption{Lax shock tube. Simulations with only applying positivity-preserving limiter on $200$ uniform cells. Snapshots are taken at $T = 1.3$. Only cell averages are plotted.}
\label{fig:Euler_lax_1D}
\end{center}
\end{figure}

\paragraph{Example~4.5 (Double rarefaction).}
This Riemann problem serves as a widely used benchmark for testing positivity-preserving schemes of solving compressible Euler equations.
We choose the computational domain $\Omega = [-1,1]$ and set the simulation end time as $T = 0.6$. 
The initial condition is prescribed as follows
\begin{align}
\transpose{(\rho^0, u^0, p^0)} = \begin{cases}
\transpose{(7, -1, 0.2)}, &\text{if}~x < 0, \\
\transpose{(7, 1, 0.2)}, &\text{if}~x \geq 0.
\end{cases} 
\end{align}
We supplement Dirichlet boundary conditions $\transpose{(\rho, m, E)} = \transpose{(7, -1, 0.2)}$ on the left end of $\Omega$ and $\transpose{(\rho, m, E)} = \transpose{(7, 1, 0.2)}$ on the right end of $\Omega$.
\par
We uniformly partition the domain $\Omega$ into $200$ cells.
Our schemes preserve positivity and conservation. Figure~\ref{fig:Euler_doub_rare} shows snapshots of the density and pressure profiles produced by $\IP^2$ and $\IP^3$ schemes at time $T = 0.6$.
\begin{figure}[ht!]
\begin{center}
\begin{tabularx}{0.85\linewidth}{@{}c@{\quad}c@{\hspace{0.75cm}}c@{}}
\begin{sideways}{$\hspace{2.75cm} \IP^2 \text{scheme} \quad$}\end{sideways} &
\includegraphics[width=0.33\textwidth]{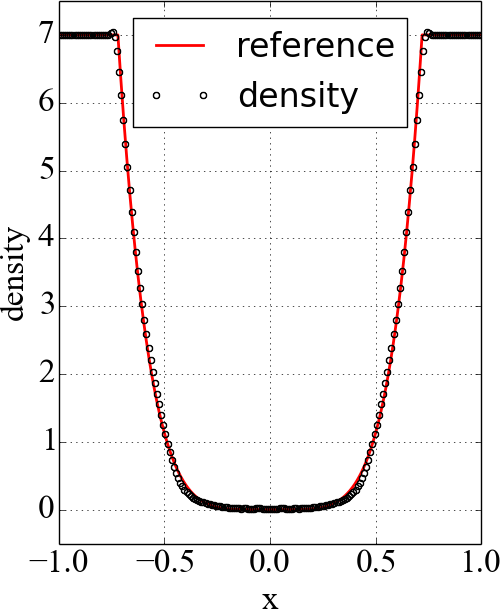} &
\includegraphics[width=0.355\textwidth]{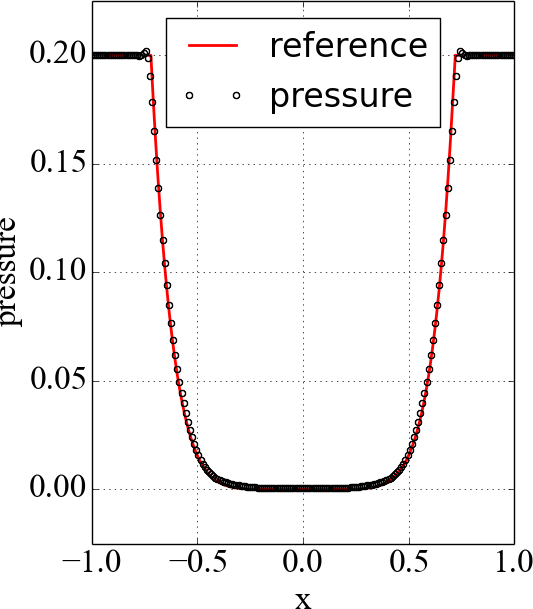} \\
\begin{sideways}{$\hspace{2.75cm} \IP^3 \text{scheme} \quad$}\end{sideways} &
\includegraphics[width=0.33\textwidth]{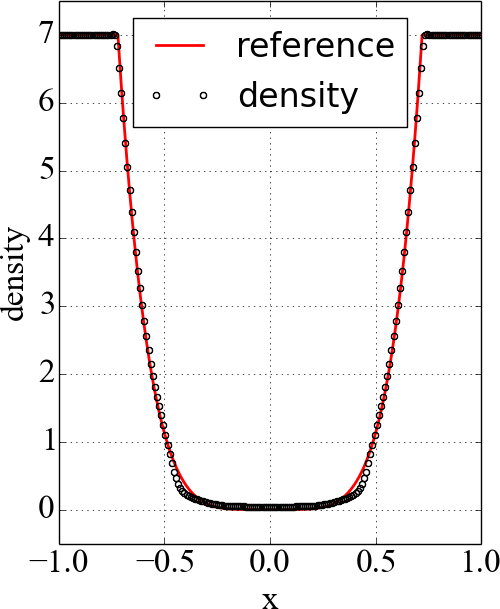} &
\includegraphics[width=0.355\textwidth]{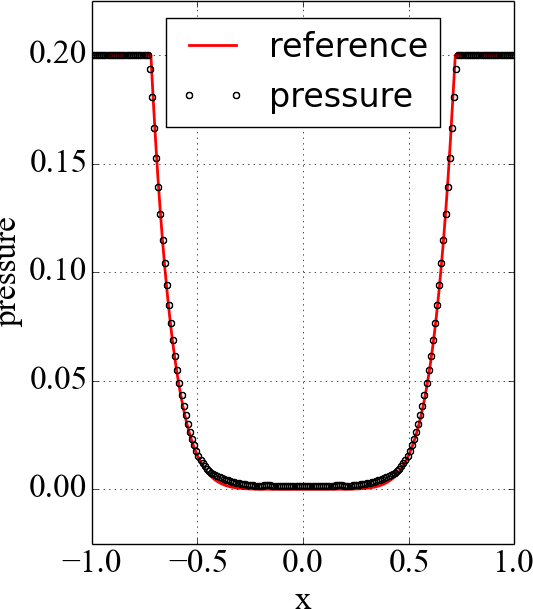} \\
\end{tabularx}
\caption{Double rarefaction. Simulations with only applying positivity-preserving limiter on $200$ uniform cells. Snapshots are taken at $T = 0.6$. Only cell averages are plotted.}
\label{fig:Euler_doub_rare}
\end{center}
\end{figure}

\subsection{Physical benchmark tests in two dimensions}
In this part, we validate the cRKDG schemes using several representative 2D benchmarks, including the traveling of Heaviside function governed by the linear advection equation for scalar conservation laws, as well as the Sedov blast wave, shock diffraction, and shock reflection diffraction governed by compressible Euler equations for nonlinear systems.
\par 
For all tests, we construct $\IP^k$ ($k = 1, 2, 3$) space basis functions using the tensor product of $k$th-order Legendre polynomials, and $\IQ^k$ space basis functions using the tensor product of $k$th-order Lagrange polynomials.
The tensor product of $(k+1)$-point Gauss quadrature is employed for numerical integration in $\IP^k$ and $\IQ^k$ schemes. 
The $(k+1)$th-order RK method is combined with the $\IP^k$ and $\IQ^k$ space discretization. 

\paragraph{Example~4.6 (Traveling Heaviside function).}
Let us consider the linear advection equation $\partial_t \rho + \div{(\rho w)} = 0$ in a squared computational domain $\Omega = [0,1]^2$. Set the simulation end time as $T = 1$. We choose velocity $w = \transpose{(0.5, 0.5)}$ in our linear advection model and solve for the unknown variable $\rho$.
The initial condition is defined as follows
\begin{align}
\rho^0(x,y) = \begin{cases}
1, &\text{if}~x+y > 0.5, \\
2, &\text{if}~x+y \leq 0.5.
\end{cases} 
\end{align}
We use Dirichlet boundary condition $\rho = 2$ on $\{|x|\leq 0.5+t, y = 0\}\cup\{|y|\leq 0.5+t, x = 0\}$.
We uniformly partition the computational domain with structured mesh of resolution $\Delta x = 1/100$.
Our schemes preserve bounds and conservation. Figure~\ref{fig:lin_adv_2D} shows snapshots of the numerical solution $\rho_h$ produced by $\IP^k$ and $\IQ^k$ ($k = 1, 2, 3$) schemes at time $T=1$.
\begin{figure}[ht!]
\begin{center}
\begin{tabularx}{\linewidth}{@{}c@{~}c@{~}c@{~}c@{}}
\includegraphics[width=0.31\textwidth]{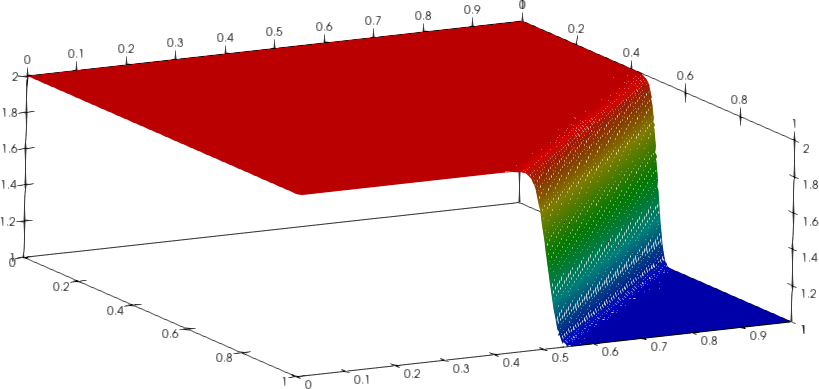} &
\includegraphics[width=0.31\textwidth]{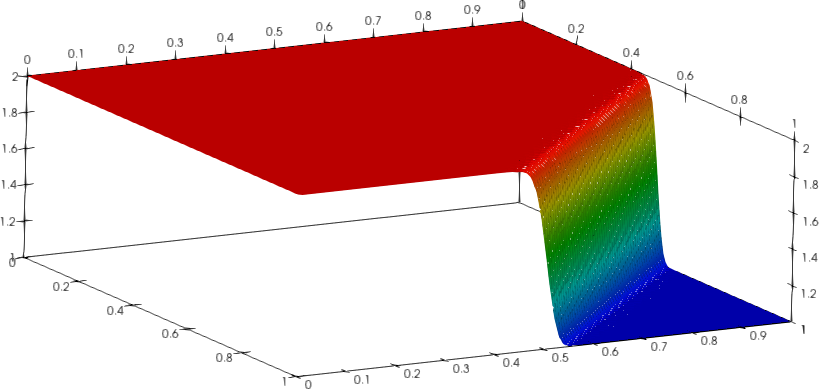} &
\includegraphics[width=0.31\textwidth]{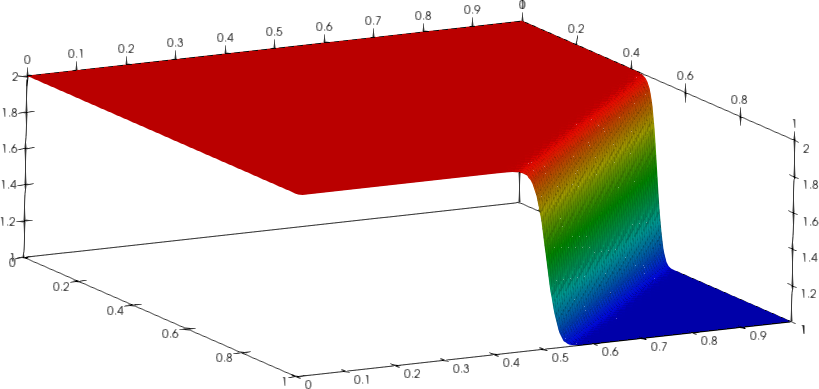} &
\includegraphics[width=0.0475\textwidth]{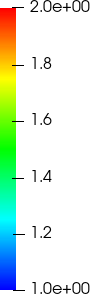} \\
$\IP^1$ scheme & 
$\IP^2$ scheme & 
$\IP^3$ scheme & ~ \\
\includegraphics[width=0.31\textwidth]{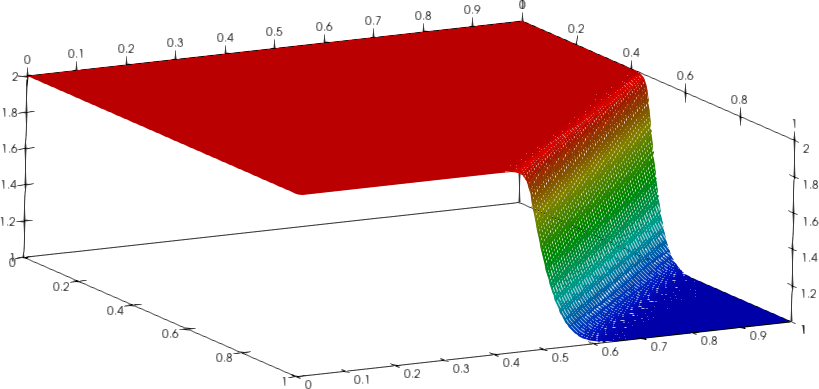} &
\includegraphics[width=0.31\textwidth]{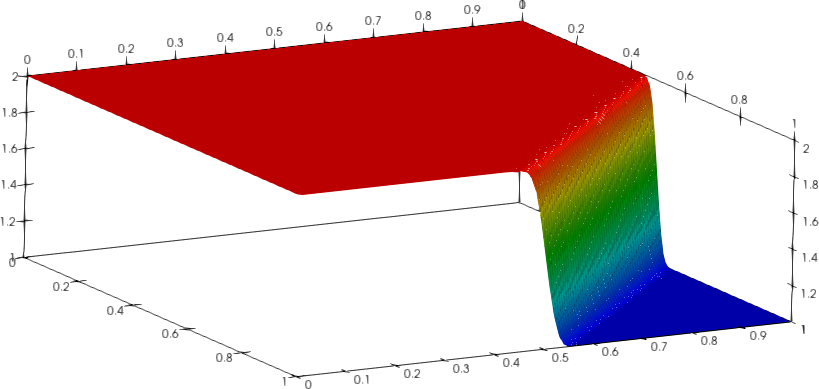} &
\includegraphics[width=0.31\textwidth]{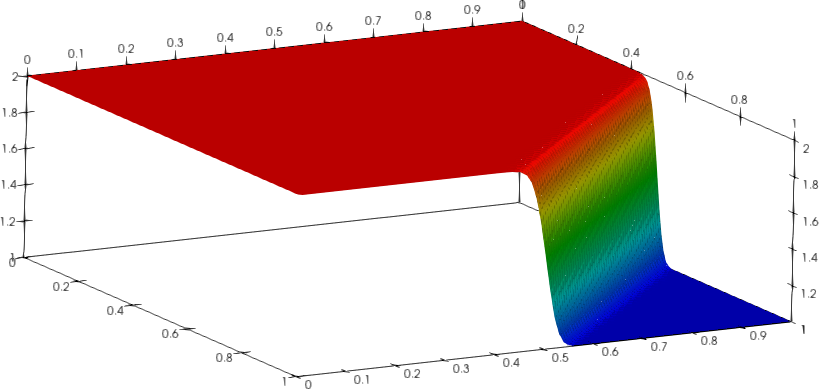} &
\includegraphics[width=0.0475\textwidth]{color_bar_adv.png} \\
$\IQ^1$ scheme & 
$\IQ^2$ scheme &
$\IQ^3$ scheme & ~ \\
\end{tabularx}
\caption{2D traveling Heaviside function. Simulations with only applying bound-preserving limiter. Snapshots are taken at $T = 1$.}
\label{fig:lin_adv_2D}
\end{center}
\end{figure}

\paragraph{Example~4.7 (Sedov blast wave).}
Let the computational domain $\Omega = [0, 1.1]^2$ and the simulation end time $T = 1$. 
We uniformly partition $\Omega$ into square cells. For $\IP^1$ and $\IQ^1$ schemes, the mesh resolution is $\Delta x = 1.1/320$. For $\IP^2$, $\IQ^2$, $\IP^3$, and $\IQ^3$ schemes, the mesh resolution is $\Delta x = 1.1/160$.
The initials are defined as piecewise constants: the density $\rho^0 = 1$, the velocity $w^0 = \transpose{(0,0)}$, and the total energy $E^0$ is set to $10^{-12}$ everywhere, except in the cell located at the lower left corner, where $0.244816/\Delta{x}^2$ is used.
Reflective boundary conditions are applied at the left and bottom boundaries, and outflow boundary conditions are applied at the right and top boundaries.
Our schemes preserve positivity and conservation. The shock location is correctly captured. Figure~\ref{fig:sedov_2D} shows snapshots of the density distribution at time $T=1$.
\begin{figure}[ht!]
\begin{center}
\begin{tabularx}{\linewidth}{@{}c@{~}c@{~}c@{~~}c@{}}
\includegraphics[width=0.3\textwidth]{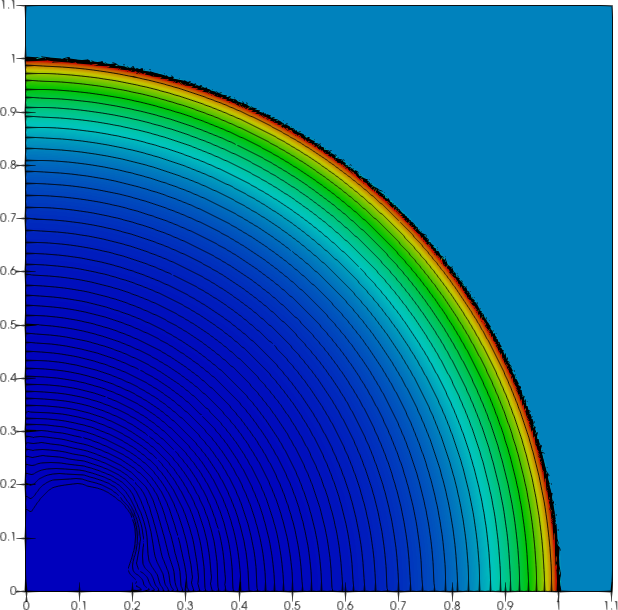} &
\includegraphics[width=0.3\textwidth]{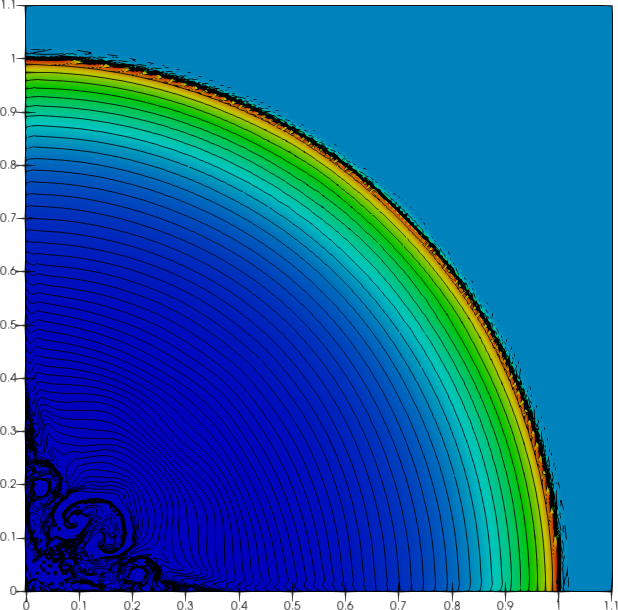} &
\includegraphics[width=0.3\textwidth]{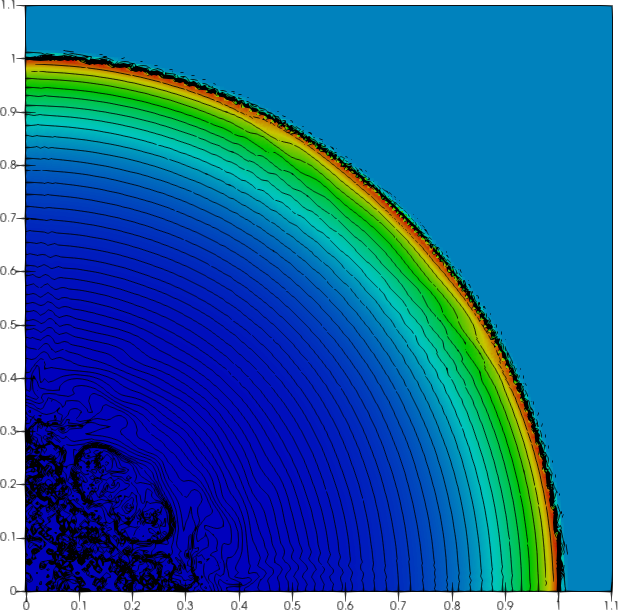} &
\includegraphics[width=0.1025\textwidth]{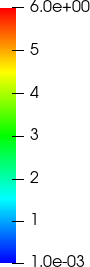}\\
$\IP^1$ scheme & 
$\IP^2$ scheme & 
$\IP^3$ scheme & ~ \\
\includegraphics[width=0.3\textwidth]{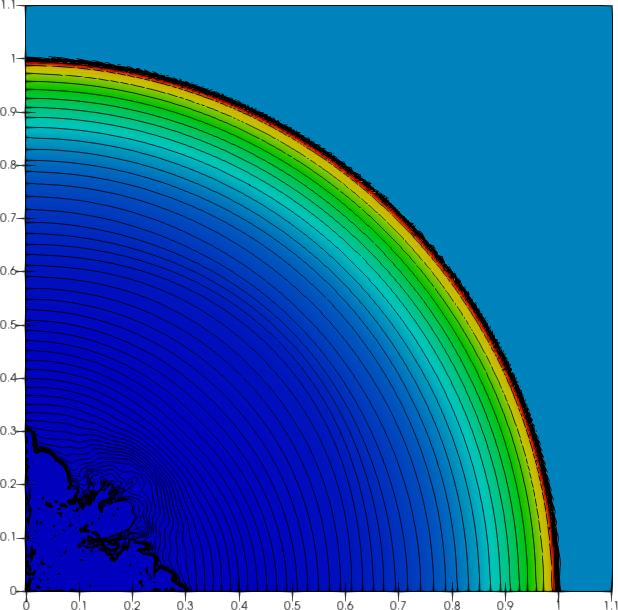} &
\includegraphics[width=0.3\textwidth]{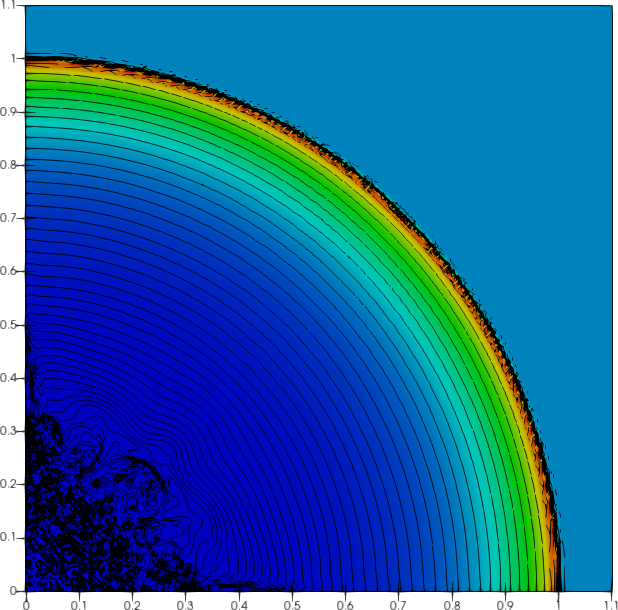} &
\includegraphics[width=0.3\textwidth]{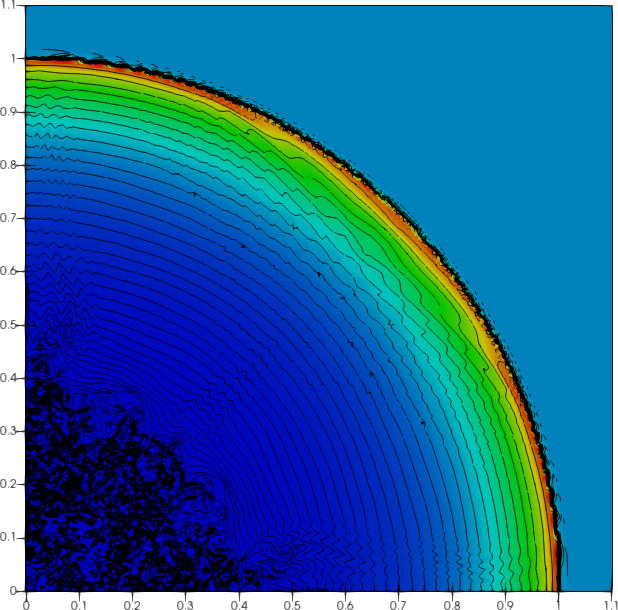} &
\includegraphics[width=0.1025\textwidth]{color_bar_sedov.png}\\
$\IQ^1$ scheme & 
$\IQ^2$ scheme & 
$\IQ^3$ scheme & ~ \\
\end{tabularx}
\caption{Sedov blast wave. Snapshots of density profile are taken at $T=1$. Plot of density: $50$ exponentially distributed contour lines of density from $0.001$ to $6$. Only the positivity-preserving limiter is applied, without any special treatment for anti-oscillation.}
\label{fig:sedov_2D}
\end{center}
\end{figure}

\paragraph{Example~4.8 (Shock diffraction).}
Let the computational domain $\Omega$ be the union of $[0, 1]\times [6, 12]$ and $[1, 13]\times [0, 12]$. We choose the simulation end time $T = 2.3$. A pure right-moving shock of Mach number $5.09$ is initially located at $x = 0.5$ and moves into undisturbed air, which has a density of $1.4$ and a pressure of $1$.
The left boundary is inflow. The top, right, and bottom boundaries are outflow. The fluid-solid boundaries $\{y = 6, 0\leq x\leq 1\}$ and $\{x = 1, 0\leq y\leq 6\}$ are reflective.
\par
We uniformly partition $\Omega$ into square cells. For $\IP^1$ and $\IQ^1$ schemes, the mesh resolution is $\Delta x = 1/96$. For $\IP^2$, $\IQ^2$, $\IP^3$, and $\IQ^3$ schemes, the mesh resolution is $\Delta x = 1/64$.
The diffraction of high-speed shocks at a sharp corner results in low density and low pressure. Our schemes preserve positivity and conservation. The shock location is correctly captured.
Figure~\ref{fig:shock_diffraction} shows snapshots of the density distribution at time $t=2.3$.
We only apply the positivity-preserving limiter without anti-oscillation. No special treatment is taken on the sharp corner.
\begin{figure}[ht!]
\begin{center}
\begin{tabularx}{\linewidth}{@{}c@{~}c@{~}c@{~}c@{}}
\includegraphics[width=0.31\textwidth]{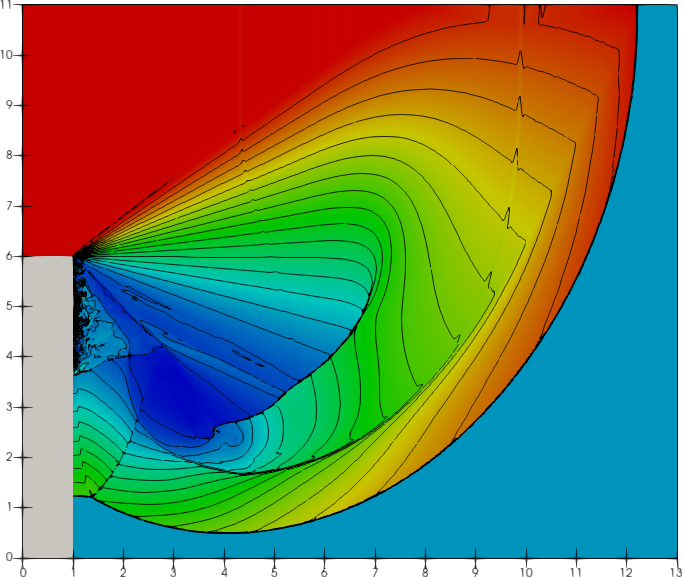} &
\includegraphics[width=0.31\textwidth]{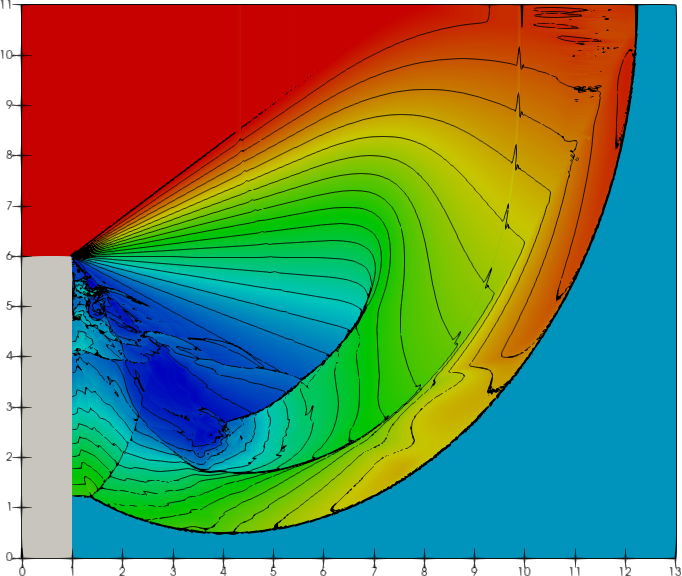} &
\includegraphics[width=0.31\textwidth]{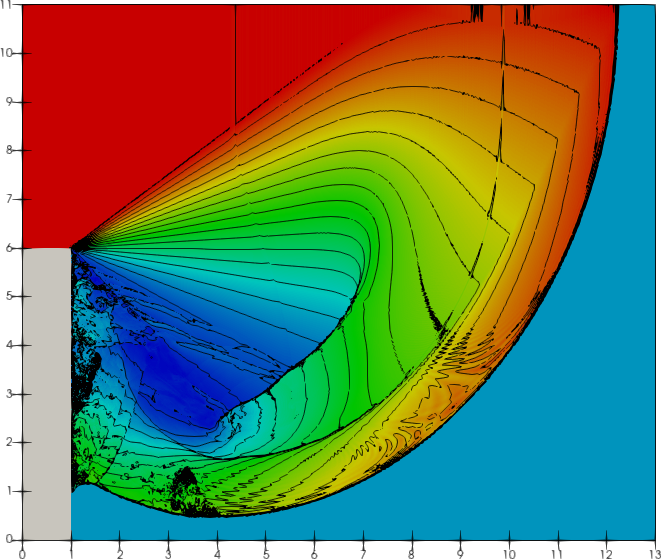} &
\includegraphics[width=0.0825\textwidth]{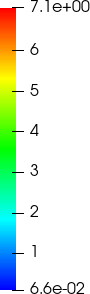} \\
$\IP^1$ scheme & 
$\IP^2$ scheme & 
$\IP^3$ scheme & ~ \\
\includegraphics[width=0.31\textwidth]{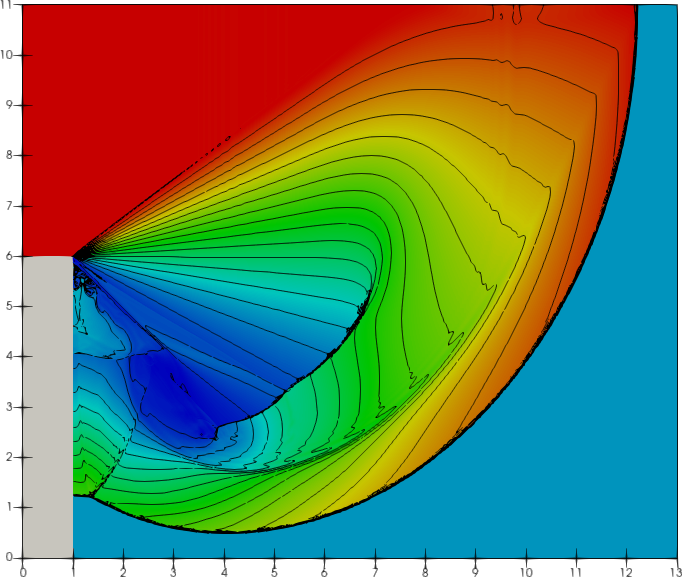} &
\includegraphics[width=0.31\textwidth]{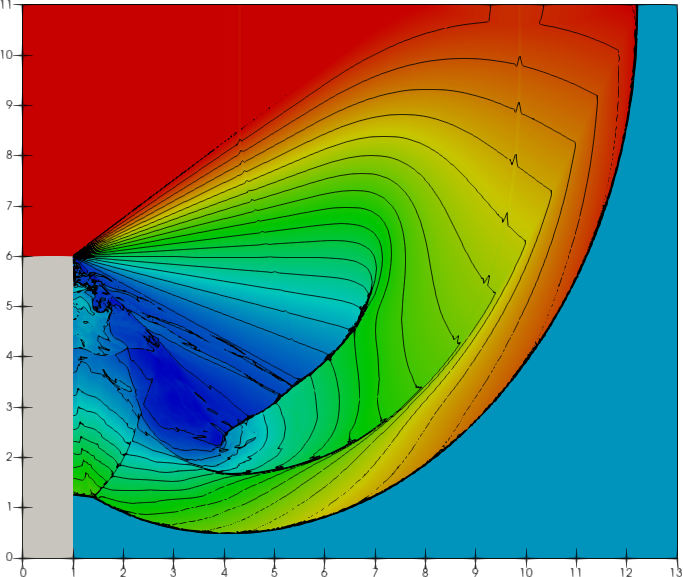} &
\includegraphics[width=0.31\textwidth]{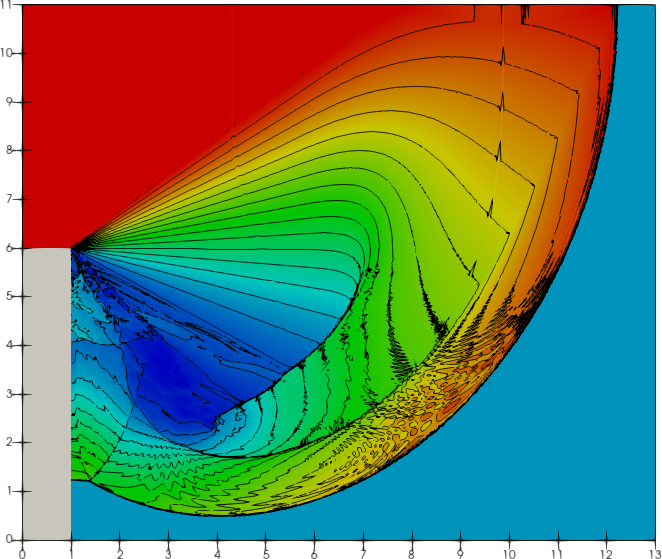} &
\includegraphics[width=0.0825\textwidth]{color_bar_sd.png} \\
$\IQ^1$ scheme & 
$\IQ^2$ scheme &
$\IQ^3$ scheme & ~ \\
\end{tabularx}
\caption{Shock diffraction. The snapshots of density profile are taken at $T = 2.3$. The gray colored region on lower-left denotes solid. Plot of density: $20$ equally spaced contour lines from $0.066227$ to $7.0668$.}
\label{fig:shock_diffraction}
\end{center}
\end{figure}

\paragraph{Example~4.9 (Shock reflection diffraction).}
The double Mach reflection of a Mach $10$ shock is a widely used benchmark for testing positivity-preserving solvers \cite{liu2023positivity,fan2022positivity}.
We consider a Mach $10$ shock moving to the right with a sixty-degree incident angle to the solid surface. Ahead of the shock is undisturbed air, characterized by a density of $1.4$ and a pressure of $1$.
Low density and low pressure appear when the shock crosses the sharp corner. In the region of shock reflection, vortices form due to Kelvin--Helmholtz instabilities.
\par
Let the computational domain $\Omega$ be the union of $[1, 4]\times [-1, 0]$ and $[0, 4]\times [0, 1]$. We choose the simulation end time $T = 0.2$. 
At initial, a Mach $10$ shock is positioned at $({1}/{6}, 0)$ and makes a $60^\circ$ angle with $x$-axis. In the post-shock region, the density is $8$, the velocity is $\transpose{(4.125\sqrt{3}, -4.125)}$, and the pressure is $116.5$.
The left boundary is inflow. The right and bottom boundaries are outflow. Part of the fluid-solid boundaries $\{y = 0, {1}/{6}\leq x\leq 1\}$ and $\{x = 1, -1\leq y\leq 1\}$ are reflective. The post-shock condition is imposed at $\{y = 0, 0\leq x < {1}/{6}\}$, where the density, velocity, and pressure are fixed in time with the initial values to make the reflected shock stick to the solid wall. On the top boundary, the flow values are set to describe the exact motion of the Mach $10$ shock.
\par
We uniformly partition $\Omega$ into square cells. For $\IP^1$ and $\IQ^1$ schemes, the mesh resolution is $\Delta x = 1/600$; for $\IP^2$ and $\IQ^2$ schemes, the mesh resolution is $\Delta x = 1/480$; and for $\IP^3$ and $\IQ^3$ schemes, the mesh resolution is $\Delta x = 1/360$.
Our schemes preserve positivity and conservation with correct shock location and well-captured rollups, see Figure~\ref{fig:shock_reflection_diffraction}.
Again, we emphasize that only the positivity-preserving limiter is applied without any anti-oscillation. No special treatment is taken on the sharp corner.
\begin{figure}[ht!]
\begin{center}
\begin{tabularx}{\linewidth}{@{}c@{}c@{~\quad~}c@{}c@{}}
\includegraphics[width=0.31\textwidth]{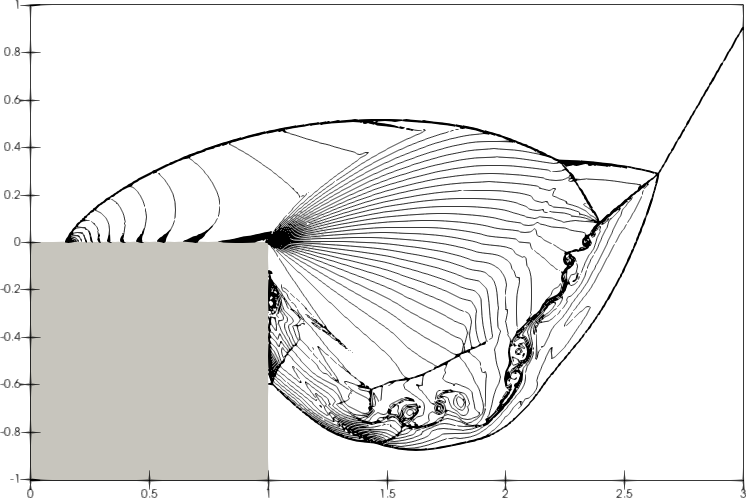} &
\includegraphics[width=0.1625\textwidth]{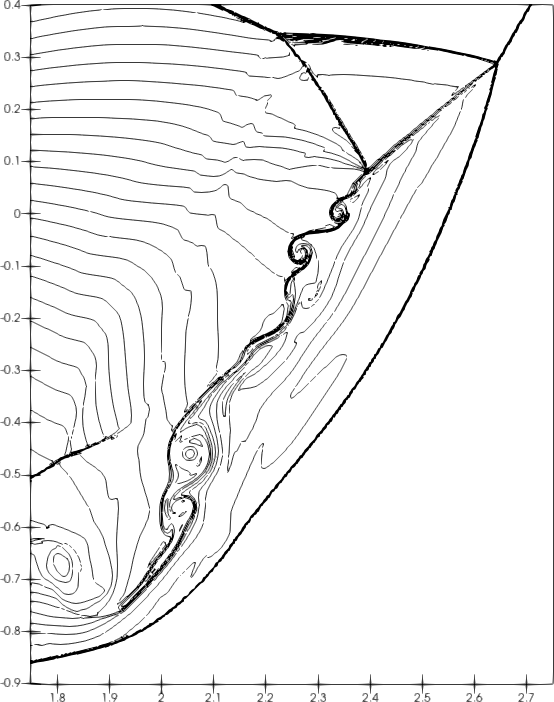} &
\includegraphics[width=0.31\textwidth]{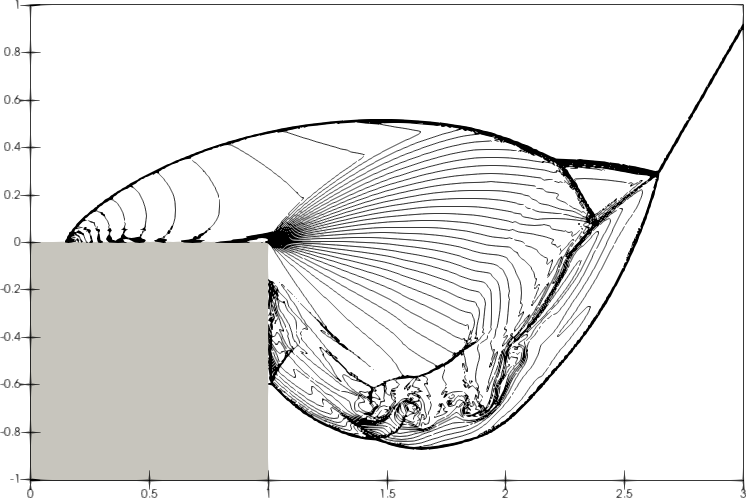} &
\includegraphics[width=0.1625\textwidth]{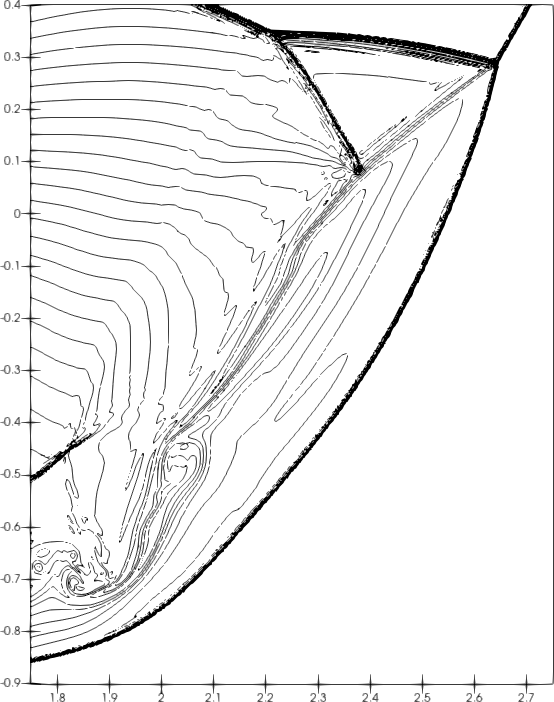} \\
\hspace{2.8cm} $\IP^1$ scheme & ~ &
\hspace{2.8cm} $\IQ^1$ scheme & ~ \\
\includegraphics[width=0.31\textwidth]{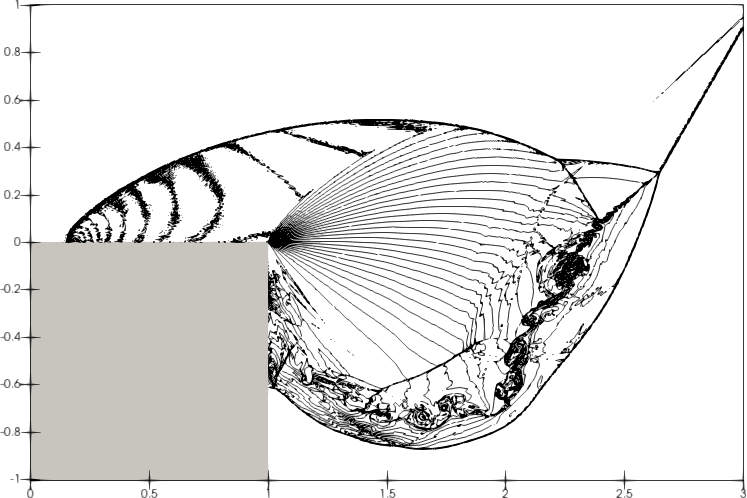} &
\includegraphics[width=0.1625\textwidth]{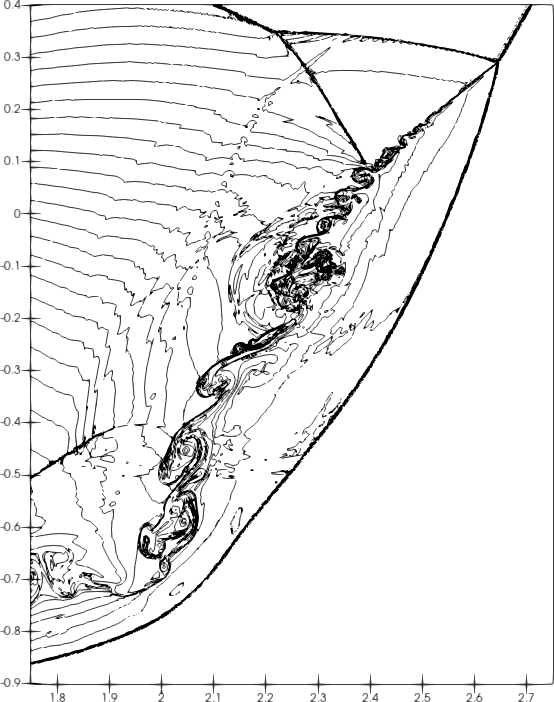} &
\includegraphics[width=0.31\textwidth]{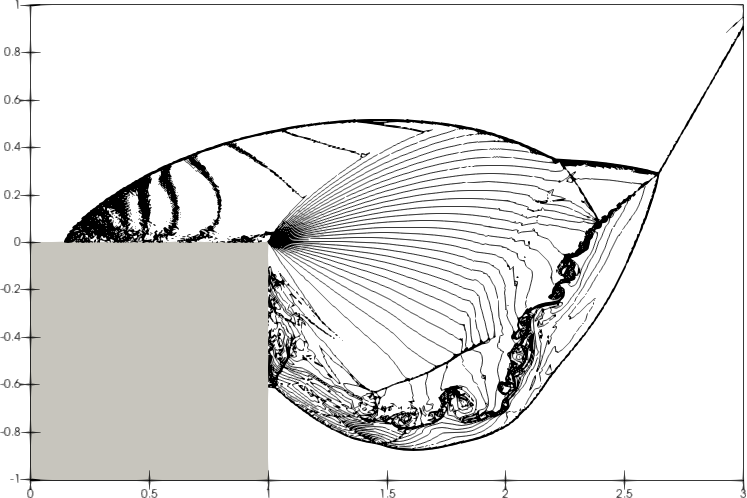} &
\includegraphics[width=0.1625\textwidth]{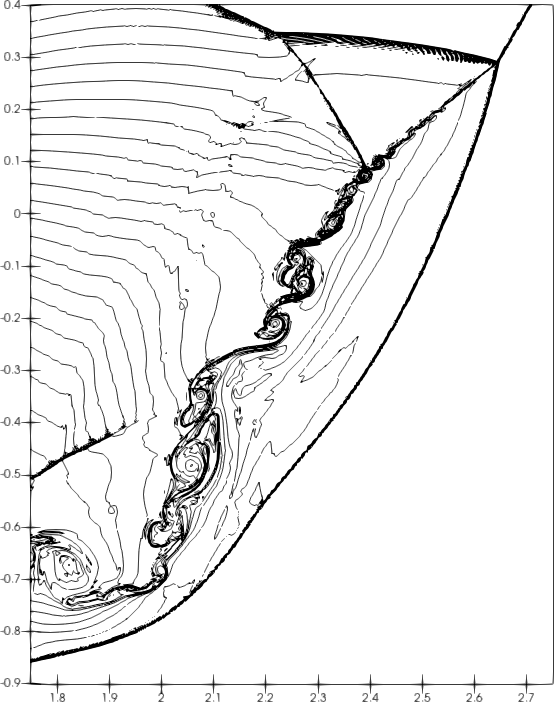} \\
\hspace{2.8cm} $\IP^2$ scheme & ~ &
\hspace{2.8cm} $\IQ^2$ scheme & ~ \\
\includegraphics[width=0.31\textwidth]{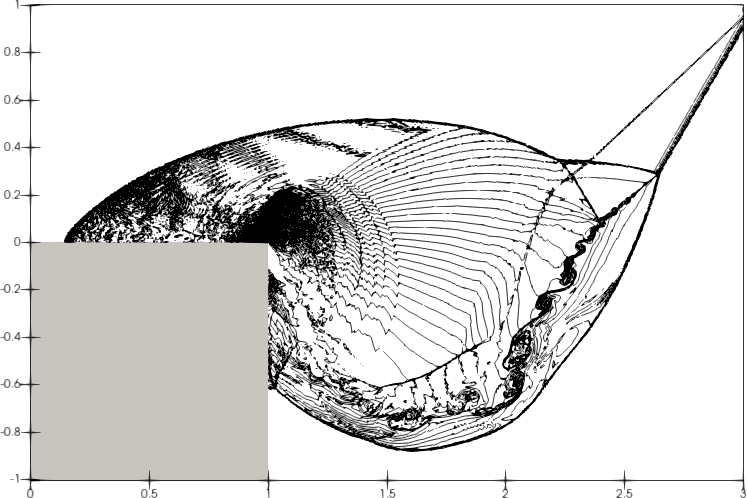} &
\includegraphics[width=0.1625\textwidth]{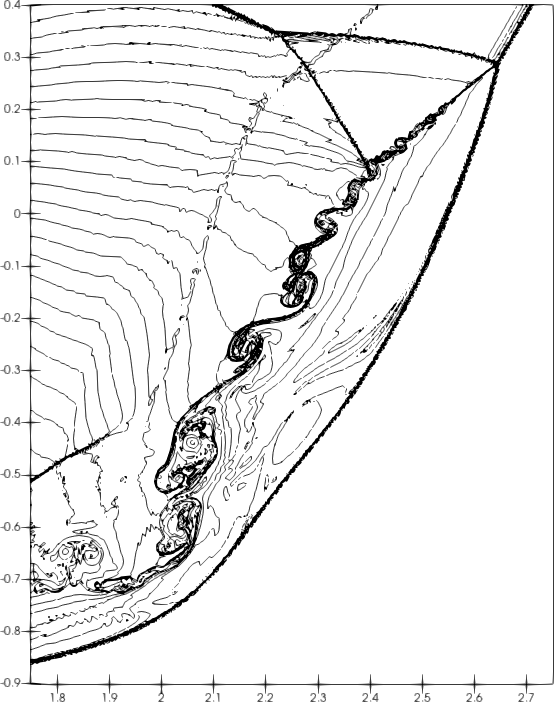} &
\includegraphics[width=0.31\textwidth]{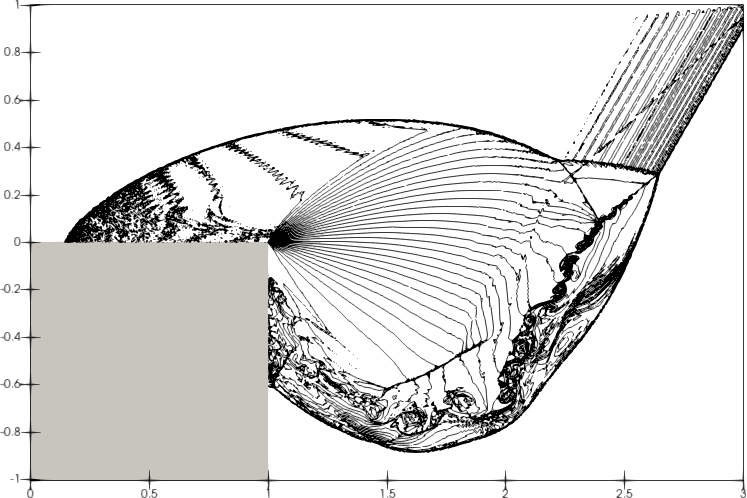} &
\includegraphics[width=0.1625\textwidth]{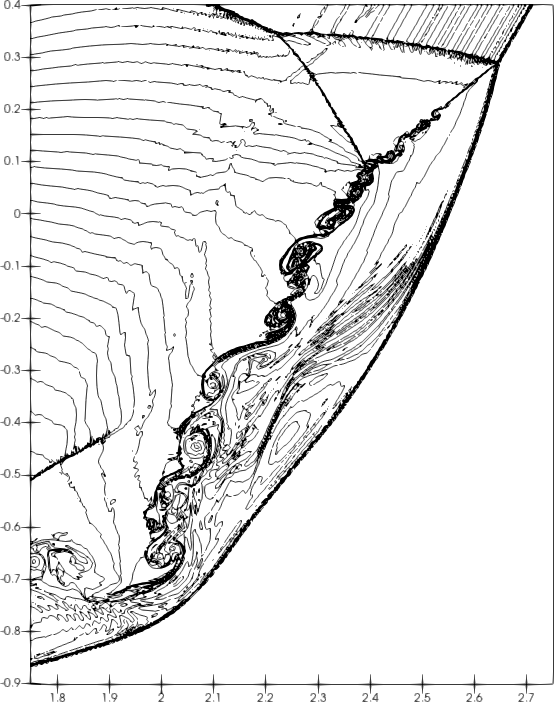} \\
\hspace{2.8cm} $\IP^3$ scheme & ~ &
\hspace{2.8cm} $\IQ^3$ scheme & ~ \\
\end{tabularx}
\caption{Mach $10$ shock reflection and diffraction. The snapshots of density profile are taken at $T = 0.2$. Only contour lines are plotted. The gray colored region on lower-left denotes solid. Plot of density: $50$ equally space contour lines from $0$ to $25$. Zoom in view of the region $[1.75, 2.75] \times [-0.9,0.4]$.}
\label{fig:shock_reflection_diffraction}
\end{center}
\end{figure}

\section{Conclusions}\label{sec:conclusion}
In this paper, we construct a bound-preserving cRKDG method for scalar conservation laws and compressible Euler equations. The main idea is to prove the weak bound-preserving property of each RK stage by representing it as convex combinations of three different forward-Euler solvers. Then a scaling limiter is applied after each stage to enforce the pointwise bounds. 
The method enforces the physical bounds for the cRKDG method, while being able to maintain its compactness, local conservation, and high-order accuracy. Our method circumvents the order barriers of SSP-RK time discretization, enabling us to construct a bound-preserving four-stage fourth-order RKDG method with compact stencils.
Numerical tests, including challenging benchmarks for 2D compressible Euler equations, are presented to demonstrate the performance of the proposed method. 

\revthr{\section*{Acknowledgment} We would like to thank the two reviewers for their valuable comments, which have helped improve the paper.}
\revone{
\appendix
\section{An alernative proof of Propsoition \ref{prop:system-pointwise}}
\label{app:GQL}
In \cite[Theorem 5.1]{wu2023geometric}, Wu and Shu proved that the invariant region for the Euler equations \eqref{eq:Euler} and \eqref{eq:Euler-2} can be equivalently formulated as
\begin{equation}\label{eq:B-GQL}
    B = \{u = (\rho, m^\top, E)^\top: u\cdot n_1 >0, u\cdot n(w_*)>0 \ \forall w_* \in \mathbb{R}^d\}. 
\end{equation}
Here $n_1 = (1,0^\top,0)^\top$ and $n(w_*) = (\frac{w_*^2}{2}, -w_*, 1)^\top$. Note for any $u\in B$, we have
\begin{equation}\label{eq:n1}
    (\pm a_0^{-1} f(u)\cdot v)\cdot n_1 = \pm \frac{w\cdot v}{a_0} u\cdot n_1<u\cdot n_1.
\end{equation}
Moreover, with the arithmetic-geometric inequality and the fact $|v|=1$, we have
\begin{equation}\label{eq:nw}
    \begin{aligned}
        (\pm a_0^{-1}f(u)\cdot v)\cdot n(w_*) =\;& \pm \frac{w\cdot v}{a_0}u\cdot n(w_*) \pm \frac{pv\cdot (w-w_*)}{a_0}\\
        \leq\;& \frac{|w\cdot v|}{a_0} u\cdot n(w_*) + \frac{\sqrt{\frac{\gamma-1}{2}\frac{p}{\rho}}}{a_0}\left(\frac{p}{\gamma-1}+\frac{1}{2}\rho |w-w_*|^2\right)\\
        =\;&\frac{|w\cdot v|+\sqrt{\frac{\gamma-1}{2}\frac{p}{\rho}}}{a_0}u\cdot n(w_*) <u\cdot n(w_*).
    \end{aligned}
\end{equation}
As a result, \eqref{eq:n1} implies that $(u\pm a_0^{-1}f(u)\cdot v)\cdot n_1>0$, and \eqref{eq:nw} implies that $(u\pm a_0^{-1}f(u)\cdot v)\cdot n(w_*)>0$. Therefore, the two conditions in \eqref{eq:B-GQL} are satisfied, which gives $u\pm a_0^{-1}f(u)\cdot v\in B$. }

\begin{revthrblock}
\section{Proof of Theorem \ref{thm:bpcrkdg}}\label{app:proofmain}
In this section, we omit the subscripts $h$ and $K$ for simplicity. 
\subsection{Main technicality: the second-order case}
The main technical challenge in the proof arises from applying the limiter after each RK stage. Before presenting the rigorous proof, we provide a brief explanation using the second-order bound-preserving cRKDG method as an example.

If no limiter is applied, Scheme 1 in \eqref{eq:scheme1} and Scheme 2 in \eqref{eq:scheme2} are equivalent.\\
\textbf{Scheme 1:}
\begin{subequations}\label{eq:scheme1}
\begin{align}
    u^{(2)} &= u^n + \frac{\Delta t}{2} \cG(u^n), \\
    u^{n+1} &= u^n + \Delta t \cF(u^{(2)}). 
    \end{align}
\end{subequations}
\textbf{Scheme 2:}
\begin{subequations}\label{eq:scheme2}
\begin{align}
    u^{(2)} &= u^n + \frac{\Delta t}{2} \cG(u^n), \\ u^{n+1} &= (2-\sqrt{3})u^n -\frac{\sqrt{3}-1}{2}\Delta t \cG(u^n) + (\sqrt{3}-1)u^{(2)} +  \Delta t \cF(u^{(2)}). 
    \end{align}
\end{subequations}
However, with the scaling limiters, denoted by $\cL$ or indicated by a tilde symbol, Scheme 1L in \eqref{eq:scheme1l} and Scheme 2L in \eqref{eq:scheme2l} are no longer equivalent.\\
\textbf{Scheme 1L:}
\begin{subequations}\label{eq:scheme1l}
\begin{alignat}{2}
    u^{(2)} &= \tilde{u}^n + \frac{\Delta t}{2} \cG(\tilde{u}^n), \quad &\tilde{u}^{(2)} &= \cL({u}^{(2)}),\\
    u^{n+1} &= \tilde{u}^n + \Delta t \cF(\tilde{u}^{(2)}),\quad &\tilde{u}^{n+1} &= \cL({u}^{n+1}). \label{eq:scheme1lb}
    \end{alignat}
\end{subequations}
\textbf{Scheme 2L:}
\begin{subequations}\label{eq:scheme2l}
\begin{align}
    u^{(2)} &= \tilde{u}^n + \frac{\Delta t}{2} \cG(\tilde{u}^n), \quad \tilde{u}^{(2)} = \cL({u}^{(2)}),\\ u^{n+1} &= (2-\sqrt{3})\tilde{u}^n -\frac{\sqrt{3}-1}{2}\Delta t \cG(\tilde{u}^n) + (\sqrt{3}-1)\tilde{u}^{(2)} +  \Delta t \cF(\tilde{u}^{(2)}), \quad \tilde{u}^{n+1} = \cL({u}^{n+1}). \label{eq:bp-limit-2}
    \end{align}
\end{subequations}
This is because, in Scheme 1L, \eqref{eq:scheme1lb} can be rewritten as 
\begin{equation}\label{eq:bp-limit}
u^{n+1} = \tilde{u}^n + \Delta t \cF(\tilde{u}^{(2)}) = (2-\sqrt{3})\tilde{u}^n -\frac{\sqrt{3}-1}{2}\Delta t \cG(\tilde{u}^n) + (\sqrt{3}-1){u}^{(2)} +  \Delta t \cF(\tilde{u}^{(2)}). 
\end{equation}
Comparing \eqref{eq:bp-limit-2} and \eqref{eq:bp-limit}, we note that \eqref{eq:bp-limit} contains $u^{(2)}$ instead of $\tilde{u}^{(2)}$ in its third-term. 

Note that Scheme 1L is the suggested implementation; however, it is the bound-preserving property of Scheme 2L that follows directly from Lemma \ref{lem:bp-stage}. 

Despite the difference between \eqref{eq:bp-limit} and \eqref{eq:bp-limit-2}, we can still prove that Scheme 1L is bound-preserving. 
Indeed, since ${u}^{(2)}$ and ${\tilde{u}}^{(2)}$  admit the same cell average, the cell averages of $u^{n+1}$ defined in \eqref{eq:bp-limit} and \eqref{eq:bp-limit-2} are the same. To be more specific, for Scheme 1L, we have
\begin{equation}
\begin{aligned}
\overline{u}^{n+1} &= \overline{\tilde{u}^n + \Delta t \cF(\tilde{u}^{(2)})} \\&=\overline{(2-\sqrt{3})\tilde{u}^n -\frac{\sqrt{3}-1}{2}\Delta t \cG(\tilde{u}^n)} +\overline{(\sqrt{3}-1){u}^{(2)} +  \Delta t \cF(\tilde{u}^{(2)})}\\
&=\overline{(2-\sqrt{3})\tilde{u}^n -\frac{\sqrt{3}-1}{2}\Delta t \cG(\tilde{u}^n)} +\overline{(\sqrt{3}-1)\tilde{u}^{(2)} +  \Delta t \cF(\tilde{u}^{(2)})}\\
&=(2-\sqrt{3})\left(\overline{\tilde{u}^n- \frac{\sqrt{3}+1}{2}\Delta t \cG(\tilde{u}^n)}\right)+(\sqrt{3}-1)\left(\overline{\tilde{u}^{(2)}+ \frac{\sqrt{3}+1}{2}\Delta t \cF(\tilde{u}^{(2)})}\right)
. 
\end{aligned}
\end{equation}
By Lemma~\ref{lem:bp-stage}, we can conclude that Scheme 1L is still bound-preserving. 

\subsection{Proof of the general case in Theorem \ref{thm:bpcrkdg}}
Now we consider the general case in Theorem \ref{thm:bpcrkdg}. With limiters, the cRKDG method can be written as 
\begin{subequations}\label{eq:rkdg-lim}
\begin{align}
u^{(i)} &= \tilde{u}^n +  \Delta t\sum_{j = 1}^{i-1}  a_{ij} \cG\left(\tilde{u}^{(j)}\right), \quad \tilde{u}^{(i)} = \cL(u^{(i)}),  \quad i = 1, 2, \cdots, s,\label{eq-rkdg1-limit}\\
u^{n+1} &= \tilde{u}^n + \Delta t \sum_{i = 1}^s b_i \cF\left(\tilde{u}^{(i)} \right), \quad \tilde{u}^{n+1} = \cL(u^{n+1}).\label{eq:rkdg-2-limit}
\end{align}
\end{subequations}

We first show that all inner stages $\tilde{u}^{(i)}$ preserve the intended bound. 

\begin{PROP}\label{prop:inner-stage}
Let $0\leq \beta_{ij}\leq 1$ and $\sum_{j=1}^{i-1} \beta_{ij} = 1$. Suppose that for the scheme without limiters, or equivalently, when $\cL = \mathcal{I}$ is the identity operator in \eqref{eq:rkdg-lim}, we have 
\begin{equation}\label{eq:convexdecomp-inner}
    u^{(i)} = u^n + \Delta t \sum_{j=1}^{i-1} a_{ij} \cG(u^{(j)}) = \sum_{j = 1}^{i-1} {\beta_{ij}}\left( u^{(j)} \pm \Delta t \mu_{ij}^{\pm\cG} \cG(u^{(j)})\right). 
\end{equation}
Then for the scheme with limiters \eqref{eq:rkdg-lim}, we have
\begin{equation}\label{eq:convexdecomp-inner-lim}
    u^{(i)} = \tilde{u}^n + \Delta t \sum_{j=1}^{i-1} a_{ij} \cG(\tilde{u}^{(j)}) = \sum_{j = 1}^{i-1} {\beta_{ij}}\left( u^{(j)} \pm \Delta t \mu_{ij}^{\pm\cG} \cG(\tilde{u}^{(j)})\right).
\end{equation}
\end{PROP}
Let us admit Proposition~\ref{prop:inner-stage} and defer its proof to the next subsection. 

By taking the cell average on both sides of \eqref{eq:convexdecomp-inner-lim} and noting the fact that the limiter doesn't alter cell averages, namely, $\overline{u}^{(j)} = \overline{\tilde{u}}^{(j)}$, we have 
\begin{equation}
        \overline{u}^{(i)} = \sum_{j = 1}^{i-1} {\beta_{ij}}\left( \overline{u}^{(j)} \pm \Delta t \mu_{ij}^{\pm\cG} \overline{\cG(\tilde{u}^{(j)})}\right)  = \sum_{j = 1}^{i-1} {\beta_{ij}}\left( \overline{\tilde{u}^{(j)} \pm \Delta t \mu_{ij}^{\pm\cG} \cG(\tilde{u}^{(j)})}\right).
\end{equation}
Therefore, by Lemma \ref{lem:bp-stage}, we have $\tilde{u}^{(i)}\in B$ under the prescribed conditions. 

Now we prove that $\tilde{u}^{n+1}$ preserves the intended bound. For the final stage, one can similarly show that 
\begin{equation}
\begin{aligned}
    u^{n+1} &= u^n + \Delta t \sum_{i = 1}^s b_i \cF\left(\tilde{u}^{(i)} \right) \\
    &= \sum_{i = 1}^s \beta_i^\cF\left(u^{(i)} + \Delta t \mu_{i}^\cF\cF(\tilde{u}^{(i)})\right) + \sum_{i = 1}^s \beta_i^{\pm\cG}\left(u^{(i)} \pm \Delta t \mu_{i}^{\pm\cG}\cG(\tilde{u}^{(i)})\right),
    \end{aligned}
\end{equation}
if the scheme without limiters admits a similar convex decomposition. Here $0\leq \beta_i^\cF,\beta_i^{\pm\cG} \leq 1$ and $\sum_{i=1}^s \beta_i^\cF + \sum_{i = 1}^s \beta_i^{\pm \cG} = 1$. After taking the cell average on both hand sides,  we have
\begin{equation}
\begin{aligned}
    \overline{u}^{n+1} &=  \sum_{i = 1}^s \beta_i^\cF\left(\overline{u^{(i)} + \Delta t \mu_{i}^\cF\cF(\tilde{u}^{(i)})}\right) + \sum_{i = 1}^s \beta_i^{\pm\cG}\left(\overline{u^{(i)} \pm \Delta t \mu_{i}^{\pm\cG}\cG(\tilde{u}^{(i)})}\right)\\
    &=  \sum_{i = 1}^s \beta_i^\cF\left(\overline{\tilde{u}^{(i)} + \Delta t \mu_{i}^\cF\cF(\tilde{u}^{(i)})}\right) + \sum_{i = 1}^s \beta_i^{\pm\cG}\left(\overline{\tilde{u}^{(i)} \pm \Delta t \mu_{i}^{\pm\cG}\cG(\tilde{u}^{(i)})}\right).
    \end{aligned}
\end{equation}
Once again, by Lemma \ref{lem:bp-stage}, we have $\tilde{u}^{n+1}\in B$ under the prescribed conditions. 

\subsection{Proof of Proposition \ref{prop:inner-stage}}
The key to the proof is the following. To derived \eqref{eq:convexdecomp-inner}, we have to use the identity
\begin{equation}\label{eq:uj}
    u^{(j)} = u^n + \Delta t \sum_{\ell = 1}^{j-1} a_{j\ell} \cG(u^{(\ell)}). 
\end{equation}
But in the scheme with limiters, this relation is replaced by 
\begin{equation}\label{eq:uj-lim}
    u^{(j)} = \tilde{u}^n + \Delta t \sum_{\ell = 1}^{j-1} a_{j\ell} \cG(\tilde{u}^{(\ell)}). 
\end{equation}
Thus we will have $u^{(j)}$ instead of $\tilde{u}^{(j)}$ for the first term inside the parenthesis on the right hand side of \eqref{eq:convexdecomp-inner-lim}. 

The detailed verification is given as follows. For the scheme without limiters, with \eqref{eq:uj} and the fact $\sum_{j=1}^{i-1} \beta_{ij} = 1$, we have 
\begin{equation}\label{eq:derive-aij}
\begin{aligned}
    \sum_{j = 1}^{i-1} {\beta_{ij}}\left( u^{(j)} \pm \Delta t \mu_{ij}^{\pm\cG} \cG(u^{(j)})\right) &= \sum_{j = 1}^{i-1}  {\beta_{ij}}\left(u^n + \Delta t\sum_{\ell = 1}^{j-1}a_{j\ell} \cG(u^{(\ell)})     \right) \pm \Delta t \sum_{j = 1}^{i-1} {\beta_{ij}}  \mu_{ij}^{\pm\cG}  \cG(u^{(j)})\\
    &= u^n + \Delta t\sum_{\ell = 1}^{i-1} \sum_{j = 1}^{\ell-1} {\beta_{i\ell}}a_{\ell j}  \cG(u^{(j)})     \pm \Delta t \sum_{j = 1}^{i-1} {\beta_{ij}}  \mu_{ij}^{\pm\cG}  \cG(u^{(j)})\\
    &= u^n + \Delta t\sum_{j = 1}^{i-2} \sum_{\ell = j+1}^{i-1} {\beta_{i\ell}}a_{\ell j}  \cG(u^{(j)})     \pm \Delta t \sum_{j = 1}^{i-1} {\beta_{ij}}  \mu_{ij}^{\pm\cG}  \cG(u^{(j)})\\
    &= u^n + \Delta t\sum_{j = 1}^{i-1} \left(\pm \beta_{ij}\mu_{ij}^{\pm\cG} + \sum_{\ell = j+1}^{i-1} {\beta_{i\ell}}a_{\ell j}\right)  \cG(u^{(j)}).
    \end{aligned}
\end{equation}
Here we have used the fact $a_{\ell,i-1} = 0$ for $j+1\leq \ell\leq i-1$ in the last equality. Therefore, comparing both hand sides of the second equality in \eqref{eq:convexdecomp-inner}, we have 
\begin{equation}
    a_{ij} = \pm \beta_{ij}\mu_{ij}^{\pm\cG} + \sum_{\ell = 1}^{i-1} {\beta_{i\ell}}a_{\ell j}. 
\end{equation}
As a result, for the scheme with limiters, by reversing the derivation in \eqref{eq:derive-aij} and applying \eqref{eq:uj-lim} in the end, we have
\begin{equation}
\begin{aligned}
    u^{(i)} &= \tilde{u}^n + \Delta t \sum_{j=1}^{i-1} a_{ij} \cG(\tilde{u}^{(j)})\\
    &= \tilde{u}^n + \Delta t\sum_{j = 1}^{i-1} \left(\pm \beta_{ij}\mu_{ij}^{\pm\cG} + \sum_{\ell = j+1}^{i-1} {\beta_{i\ell}}a_{\ell j}\right)  \cG(\tilde{u}^{(j)})\\
    &= \sum_{j = 1}^{i-1} {\beta_{ij}}\left( \left( \tilde{u}^n + \Delta t \sum_{\ell = 1}^{j-1} a_{j\ell} \cG(\tilde{u}^{(\ell)})\right) \pm \Delta t \mu_{ij}^{\pm\cG} \cG(\tilde{u}^{(j)})\right)\\
    &= \sum_{j = 1}^{i-1} {\beta_{ij}}\left( u^{(j)} \pm \Delta t \mu_{ij}^{\pm\cG} \cG(\tilde{u}^{(j)})\right).
\end{aligned}
\end{equation}
This completes the proof of Proposition \ref{prop:inner-stage}.
\end{revthrblock}
\bibliographystyle{abbrv}

\end{document}